\documentclass{article}

\usepackage{amsmath,amsthm,bm}
\usepackage{authblk}
\usepackage{enumerate}
\usepackage{comment}
\usepackage{amssymb}
\usepackage{pgfplots}
\usepackage{subcaption}
\usepackage{graphicx}
\usepackage{url}
\usetikzlibrary{arrows,positioning,calc,patterns}

\newtheorem{thm}{Theorem}[section]
\newtheorem{cor}[thm]{Corollary}

\newtheorem{prop}[thm]{Proposition}
\newtheorem{lem}[thm]{Lemma}
\newtheorem{defin}[thm]{Definition}

\makeatletter
\newcommand*{\house}[1]{%
  \mathord{%
    \mathpalette\@house{#1}%
  }%
}
\newcommand*{\@house}[2]{%
  \dimen@=\fontdimen8 %
      \ifx#1\scriptscriptstyle\scriptscriptfont
      \else\ifx#1\scriptstyle\scriptfont
      \else\textfont\fi\fi
      3 %
  \sbox0{%
    $#1%
      \vrule width\dimen@\relax
      \overline{%
        \kern2\dimen@
        \begingroup 
          #2%
        \endgroup
        \kern2\dimen@
      }%
      \vrule width\dimen@\relax
      \mathsurround=1.5\dimen@ 
    $%
  }%
  \ht0=\dimexpr\ht0-\dimen@\relax
  \dp0=\dimexpr\dp0+2\dimen@\relax
  \vbox{%
    \kern\dimen@ 
    \copy0 %
  }%
}

\newcommand{\CC}{\mathbb C}

\newcommand{\R}{\mathbb R}
\newcommand{\Z}{\mathbb Z}
\newcommand{\Q}{\mathbb Q}

\DeclareMathOperator{\Aut}{\operatorname{Aut}}
\DeclareMathOperator{\orb}{\operatorname{orb}}

\DeclareMathOperator{\Res}{\operatorname{Res}}

\DeclareMathOperator{\GL}{\operatorname{GL}}

\begin{document}

\title{On the Generalization of the Gap Principle}

\author{Anton Mosunov}
\affil{University of Waterloo}

\date{}

\maketitle

\begin{abstract}
Let $\alpha$ be a real algebraic number of degree $d \geq 3$ and let $\beta \in \mathbb Q(\alpha)$ be irrational. Let $\mu$ be a real number such that $(d/2) + 1 < \mu < d$ and let $C_0$ be a positive real number. We prove that there exist positive real numbers $C_1$ and $C_2$, which depend only on $\alpha$, $\beta$, $\mu$ and $C_0$, with the following property. If $x_1/y_1$ and $x_2/y_2$ are rational numbers in lowest terms such that
$$
H(x_2, y_2) \geq H(x_1, y_1) \geq C_{1}
$$
and
$$
\left|\alpha - \frac{x_1}{y_1}\right| < \frac{C_0}{H(x_1, y_1)^\mu}, \quad \left|\beta - \frac{x_2}{y_2}\right| < \frac{C_0}{H(x_2, y_2)^\mu},
$$
then either $H(x_2, y_2) > C_{2}^{-1} H(x_1, y_1)^{\mu - d/2}$, or there exist integers $s, t, u, v$, with $sv - tu \neq 0$, such that
$$
\beta = \frac{s\alpha + t}{u\alpha + v} \quad \text{and} \quad \frac{x_2}{y_2} = \frac{sx_1 + ty_1}{ux_1 + vy_1},
$$
or both. Here $H(x, y) = \max(|x|, |y|)$ is the height of $x/y$. Since $\mu - d/2$ exceeds one, our result demonstrates that, unless $\alpha$ and $\beta$ are connected by means of a linear fractional transformation with integer coefficients, the heights of $x_1/y_1$ and $x_2/y_2$ have to be exponentially far apart from each other. An analogous result is established in the case when $\alpha$ and $\beta$ are $p$-adic algebraic numbers.
\end{abstract}

\section{Introduction}
The theory of Diophantine approximation concerns the question of how well real numbers can be approximated by rationals, and its variations. If $\alpha$ is a real number and $x/y$ is a rational number, with $x, y \in \mathbb Z$ and $y \geq 1$, then the quality of approximation of $\alpha$ by $x/y$ can be measured by means of a quantity $\mu$ such that the inequality
\begin{equation} \label{eq:diophantine_inequality}
\left|\alpha - \frac{x}{y}\right| < \frac{1}{y^\mu}
\end{equation}
is satisfied. The larger $\mu$ is, the better the approximation of $x/y$ with respect to $\alpha$ is. It was observed by Dirichlet that for $\mu = 2$ the inequality above can be achieved for infinitely many integers $x$ and $y$, as long as $\alpha$ is real and irrational. On the other hand, Liouville pointed out that, if $\alpha$ is an irrational algebraic number of degree $d$ and $\mu > d$, then (\ref{eq:diophantine_inequality}) has only finitely many solutions in integers $x$ and $y$ with $y \geq 1$. In other words, algebraic numbers cannot be approximated by rationals too well.

It is not difficult to count distinct $x/y$ satisfying (\ref{eq:diophantine_inequality}), with $y$ varying in a fixed range. Indeed, if it so happens that $C_1 \leq y_1 < y_2 \leq C_2$, then the fact that $x_1/y_1 \neq x_2/y_2$ yields
$$
\frac{1}{y_1y_2} \leq \left|\frac{x_1}{y_1} - \frac{x_2}{y_2}\right| \leq \left|\alpha - \frac{x_1}{y_1}\right| + \left|\alpha - \frac{x_2}{y_2}\right| < \frac{1}{y_1^\mu} + \frac{1}{y_2^\mu} < \frac{2}{y_1^\mu},
$$
resulting in the inequality
$$
2y_2 > y_1^{\mu - 1},
$$
which is known as the \emph{gap principle}. For $\mu > 2$ this inequality states that, if two distinct rationals satisfy (\ref{eq:diophantine_inequality}), then their denominators must be exponentially far apart from each other.

Unfortunately, as the quantity $C_2$ can be arbitrarily large, the gap principle itself does not allow us to count the number of distinct solutions to (\ref{eq:diophantine_inequality}). However, it was established by Thue \cite{thue} that, when $\alpha$ is an irrational algebraic number of degree $d \geq 3$ and $(d/2) + 1 < \mu < d$, then there exist computable positive real numbers $C_1$ and $\eta > 1$, which depend only on $\alpha$ and $\mu$, such that every solution $x_i/y_i$ with $C_1 \leq y_1 < \ldots < y_\ell$ satisfies $y_i < y_1^\eta$. This phenomenon is known as the \emph{Thue-Siegel principle} and it was vastly generalized by Bombieri and Mueller \cite{bombieri-mueller}. When combined with the gap principle, the Thue-Siegel principle enables us to count the number of solutions $x/y$ to (\ref{eq:diophantine_inequality}) such that $y \geq C_1$.

For a rational number $x/y$ in lowest terms, let $H(x, y) = \max(|x|, |y|)$ denote the \emph{height} of $x/y$. In this article, we generalize the gap principle as follows. Notice that the positive real numbers $C_1, C_2, \ldots$ occurring throughout the article are all computable. 

\begin{thm} \label{thm:archimedean_gap_principle}
\emph{(A generalized Archimedean gap principle)} Let $\alpha$ be a real algebraic number of degree $d \geq 3$ over $\Q$ and let $\beta$ be irrational and in $\Q(\alpha)$. Let $\mu$ be a real number such that $(d/2) + 1 < \mu < d$ and let $C_0$ be a positive real number. There exist positive real numbers $C_1$ and $C_2$, which depend only on $\alpha$, $\beta$, $\mu$ and $C_0$, with the following property. If $x_1/y_1$ and $x_2/y_2$ are rational numbers in lowest terms such that $H(x_2, y_2) \geq H(x_1, y_1) \geq C_{1}$ and
$$
\left|\alpha - \frac{x_1}{y_1}\right| < \frac{C_0}{H(x_1, y_1)^\mu}, \quad \left|\beta - \frac{x_2}{y_2}\right| < \frac{C_0}{H(x_2, y_2)^\mu},
$$
then at least one of the following holds:
\begin{itemize}
\item  $H(x_2, y_2) > C_{2}^{-1} H(x_1, y_1)^{\mu - d/2}$;
 \item There exist integers $s, t, u, v$, with $sv - tu \neq 0$, such that
$$
\beta = \frac{s\alpha + t}{u\alpha + v} \quad \text{and} \quad \frac{x_2}{y_2} = \frac{sx_1 + ty_1}{ux_1 + vy_1}.
$$
\end{itemize}
\end{thm}

Since the exponent $\mu - d/2$ exceeds one, our result demonstrates that, unless $\alpha$ and $\beta$ are connected by means of a linear fractional transformation with integer coefficients, the heights of $x_1/y_1$ and $x_2/y_2$ have to be exponentially far apart from each other.

Next, let $|\quad|_p$ denote the $p$-adic absolute value on the field of $p$-adic numbers $\Q_p$, normalized so that $|p|_p = p^{-1}$. An analogous result for $p$-adic algebraic numbers is as follows.

\begin{thm} \label{thm:non-archimedean_gap_principle}
\emph{(A generalized non-Archimedean gap principle)} Let $p$ be a rational prime. Let $\alpha \in \Q_p$ be a $p$-adic algebraic number of degree $d \geq 3$ over $\Q$ and let $\beta$ be irrational and in $\Q(\alpha)$. Let $\mu$ be a real number such that $(d/2) + 1 < \mu < d$ and let $C_0$ be a positive real number. There exist positive real numbers $C_{3}$ and $C_{4}$, which depend only on $\alpha$, $\beta$, $\mu$ and $C_0$, with the following property. If $x_1/y_1$ and $x_2/y_2$ are rational numbers in lowest terms such that $H(x_2, y_2) \geq H(x_1, y_1) \geq C_{3}$ and
$$
\left|y_1\alpha - x_1\right|_p < \frac{C_0}{H(x_1, y_1)^\mu}, \quad \left|y_2\beta - x_2\right|_p < \frac{C_0}{H(x_2, y_2)^\mu},
$$
then at least one of the following holds:
\begin{itemize}
\item  $H(x_2, y_2) > C_{4}^{-1} H(x_1, y_1)^{\mu - d/2}$;
 \item There exist integers $s, t, u, v$, with $sv - tu \neq 0$, such that
$$
\beta = \frac{s\alpha + t}{u\alpha + v} \quad \text{and} \quad \frac{x_2}{y_2} = \frac{sx_1 + ty_1}{ux_1 + vy_1}.
$$
\end{itemize}
\end{thm}

We apply our result to establish an absolute bound on the number of large primitive solutions of certain Thue inequalities. A \emph{Thue inequality} is an inequality of the form
\begin{equation} \label{eq:thue-inequality}
0 < |F(x, y)| \leq m,
\end{equation}
where $m$ is a positive integer and $F \in \Z[x, y]$ is an irreducible binary form of degree $d \geq 3$. A solution $(x, y) \in \Z^2$ to the above inequality is called \emph{primitive} when $x$ and $y$ are coprime. In \cite[Theorem 1]{gyory}, Gy\H{o}ry proved that there exists a positive real number $Y_0$, which depends only on $m$ and $F$, such that the number of primitive solutions $(x, y)$ to (\ref{eq:thue-inequality}) with $H(x, y) \geq Y_0$ does not exceed $25d$ (here the solutions $(x, y)$ and $(-x, -y)$ are regarded as the same). Using Theorem \ref{thm:archimedean_gap_principle}, we can improve Gy\H{o}ry's result in the case when $F$ is irreducible and the field extension $\Q(\alpha)/\Q$ is Galois, where $\alpha$ is a root of $F(x, 1)$.

To state our result, we need to introduce the notion of \emph{enhanced automorphism group} of $F$. For a $2 \times 2$ matrix $M = \left(\begin{smallmatrix}s & u\\t & v\end{smallmatrix}\right)$, with complex entries, define the binary form $F_M$ by
$$
F_M(x, y) = F(sx + uy, tx + vy).
$$
Let $\overline{\Q}$ denote the algebraic closure of the rationals and let $K$ be a field containing $\mathbb Q$. We say that a matrix $M = \left(\begin{smallmatrix}s & u\\t & v\end{smallmatrix}\right) \in \operatorname{M}_2(K)$ is a \emph{$K$-automorphism of $F$} (resp., $|F|$) if $F_M = F$ (resp., $F_M = \pm F$). The set of all $K$-automorphisms of $F$ (resp., $|F|$) is denoted by $\Aut_K F$ (resp., $\operatorname{Aut}_K |F|$). We define
\begin{equation} \label{eq:G}
\Aut' |F| = \left\{\frac{1}{\sqrt{|sv - tu|}}\begin{pmatrix}s & u\\t & v\end{pmatrix} \colon s, t, u, v \in \Z\right\} \cap \Aut_{\overline{\Q}} |F|
\end{equation}
and refer to it as the \emph{enhanced automorphism group} of $F$. One can verify that $\operatorname{Aut}' |F|$ is a group. In Section \ref{sec:automorphisms} we will show that, under the conditions on $F$ specified above, it is finite and contains at most $24$ elements. We prove the following.

\begin{thm} \label{thm:thue_large}
Let $F \in \Z[x, y]$ be an irreducible binary form of degree \mbox{$d \geq 3$}. Let $\alpha$ be a root of $F(x, 1)$ and assume that the field extension $\Q(\alpha)/\Q$ is Galois. For a positive integer $m$ consider the Thue inequality (\ref{eq:thue-inequality}). Let $\mu$ be a real number such that $(d/2) + 1 < \mu < d$. There exists a positive real number $C_5$, which depends only on $m$, $F$ and $\mu$, such that the number of primitive solutions $(x, y)$ to (\ref{eq:thue-inequality}) with $H(x, y) \geq C_5$ does not exceed
$$
\#\Aut'|F|\cdot \left\lfloor 1 + \frac{11.51 + 1.5\log d + \log \mu}{\log(\mu - d/2)}\right\rfloor.
$$
Here the solutions $(x, y)$ and $(-x, -y)$ are regarded as the same.
\end{thm}

Let $\mu = (3d + 2)/4$. Then the function
$$
f(d) = 1 + \frac{11.51+ 1.5\log d + \log((3d + 2)/4)}{\log((d + 2)/4)}
$$
is monotonously decreasing on the interval $[3, \infty)$. To see that this is the case, it is sufficient to prove that $g(x) = \frac{\log x}{\log((x + 2)/4)}$ and $h(x) = \frac{\log((3x + 2)/4)}{\log((x + 2)/4)}$ are monotonously decreasing on the specified interval. We leave it as an exercise to the reader to prove that the derivatives of $g(x)$ and $h(x)$ take negative values when evaluated at any $x_0 \geq 3$. Since $f(3) \approx 64.5$, we can use the upper bound $\#\Aut'|F| \leq 24$ established in Lemma \ref{lem:G_is_finite} as well as Theorem \ref{thm:thue_large} to conclude that the number of primitive solutions $(x, y)$ to (\ref{eq:thue-inequality}) such that $H(x, y) \geq C_5$ does not exceed $24 \cdot \lfloor f(3)\rfloor = 1536$ when $d \geq 3$. Furthermore, since $f(10^{14}) < 4$ and $\lim\limits_{d \rightarrow \infty} f(d) = 3.5$, we can also conclude that it does not exceed $24 \cdot \lfloor f(10^{14})\rfloor = 72$ when $d \geq 10^{14}$. While it is an interesting task to compare the value of $C_5$ to the quantity $Y_L$ in \cite{mueller-schmidt} (see equation 2.9) or to the quantity $Y_0$ in \mbox{\cite[Theorem 1]{gyory}}, it lies outside the scope of this article.

The article is structured as follows. In Section \ref{sec:theory} we outline a number of auxiliary results, which are used in the later sections. We recommend the reader to skip this section and use it as a reference. In Section \ref{sec:minimal_pairs} we introduce the notion of a \emph{minimal pair} $P, Q \in \Z[x]$ for a tuple of algebraic numbers $(\alpha, \beta) \in \overline{\Q} \times (\Q(\alpha) \setminus \Q)$, and summarize the properties of minimal pairs. Minimal pairs enable us to construct a nonzero polynomial $R(x, y) = P(x) + yQ(x)$, which vanishes at the point $(\alpha, \beta)$. When $R(x, y)$ does not vanish at the rational point $\left(x_1/y_1,\ x_2/y_2\right)$, establishing the gap principle is a rather easy task. In \mbox{Section \ref{sec:gap-principle-with-vanishing}}, we prove that \emph{despite the vanishing} of $R(x, y)$ at $\left(x_1/y_1,\ x_2/y_2\right)$, it is still possible to prove that the heights of $(x_1, y_1)$ and $(x_2, y_2)$ are exponentially far apart, provided that $\alpha$ and $\beta$ are not connected by means of a linear fractional transformation with integer coefficients. In Sections \ref{sec:generalized_archimedean_gap_principle} and \ref{sec:generalized_non-archimedean_gap_principle} we prove Theorems \ref{thm:archimedean_gap_principle} and \ref{thm:non-archimedean_gap_principle}, respectively. In Section \ref{sec:automorphisms} we investigate the properties of the enhanced automorphism group $\Aut' |F|$. Finally, in Section \ref{sec:thue-inequality} we prove \mbox{Theorem \ref{thm:thue_large}}.

\section{Auxiliary Results} \label{sec:theory}

This section contains several definitions and results, which we utilize in the remaining part of the article. We recommend the reader to skip this section and refer to it when reading the proofs outlined in sections that follow it.

Let us begin with the number of definitions. For an arbitrary polynomial $R \in \Z[x_1, x_2, \ldots, x_n]$, we let $H(R)$ denote the maximum of Archimedean absolute values of its coefficients, and refer to this quantity as a \emph{height} of $R$. For an algebraic number $\alpha$ with the minimal polynomial $f$, we write $H(\alpha) = H(f)$. For a point $(x_1, x_2, \ldots, x_n) \in \CC^n$, we define
$$
H(x_1, x_2, \ldots, x_n) = \max\limits_{i = 1, 2, \ldots, n}\left\{|x_i|\right\}
$$
and refer to this quantity as the \emph{height} of $(x_1, x_2, \ldots, x_n)$.

In this section, as well as all the subsequent ones, we write
$$
D_{i,j} = \frac{1}{i!j!}\frac{\partial^{i + j}}{\partial X^i\partial Y^j} \quad \text{and} \quad D_i = \frac{1}{i!}\frac{d^i}{dX^i}.
$$

\begin{lem} \label{lem:liouville}
\emph{(Liouville's Theorem)} Let $\alpha \in \CC$ be an algebraic number of degree $d$ over $\Q$. There exists a positive number $C_6$, which depends only on $\alpha$, such that for all integers $x$ and $y$, with $y \neq 0$ and $x \neq y\alpha$, the inequality
\begin{equation} \label{eq:liouville_inequality}
\left|\alpha - \frac{x}{y}\right| \geq \frac{C_6}{H(x, y)^d}
\end{equation}
holds.
\end{lem}

\begin{proof}
See \cite[Theorem 1E]{schmidt3}.
\end{proof}

\begin{lem} \label{lem:p-adic_liouville}
\emph{($p$-adic Liouville's Theorem)} Let $p$ be a rational prime and $\alpha \in \Q_p$ a $p$-adic algebraic number of degree $d$ over $\Q$. There exists a positive number $C_7$, which depends only on $\alpha$, such that for all integers $x$ and $y$, with $x \neq y\alpha$, the inequality
\begin{equation} \label{eq:p-adic_liouville_inequality}
|y\alpha - x|_p \geq \frac{C_7}{H(x, y)^d}
\end{equation}
holds.
\end{lem}

\begin{proof}
Let
$$
f(x) = c_dx^d + \cdots + c_1x + c_0
$$
be the minimal polynomial of $\alpha$ and let $F(x, y) = y^df(x/y)$ be its associated binary form. Since $f(\alpha) = 0$, it follows from Taylor's Theorem that
\begin{align*}
F(x ,y)
& = (x - \alpha y)\sum\limits_{i = 1}^dD_if(\alpha)(x - \alpha y)^{i - 1}y^{d - i}\\
& = (x - \alpha y)\sum\limits_{i = 1}^d\frac{D_if(\alpha)}{c_d^{i - 1}}(c_dx - c_d\alpha y)^{i - 1}y^{d - i}.
\end{align*}
Since $c_d\alpha$ and  $c_d^{d - i}D_if(\alpha)$ are algebraic integers, their $p$-adic absolute values do not exceed one, so
\begin{align*}
\left|F(x, y)\right|_p
& \leq |y\alpha - x|_p\cdot \max\limits_{i = 1, \ldots, d}\left\{\left|\frac{D_if(\alpha)}{c_d^{i - 1}}\right|_p\right\}\\
& = |y\alpha - x|_p\cdot \max\limits_{i = 1, \ldots, d}\left\{\left|\frac{c_d^{d - i}D_if(\alpha)}{c_d^{d - 1}}\right|_p\right\}\\
& \leq |y\alpha - x|_p \cdot |c_d|_p^{-{d+1}}.
\end{align*}
Since $x \neq y\alpha$, it must be the case that $F(x, y) \neq 0$. By the product formula, the following trivial lower bound holds:
$$
\left|F(x, y)\right|_p \geq \frac{1}{\left|F(x, y)\right|} \geq \frac{1}{(d + 1)H(\alpha)H(x, y)^d}.
$$
The result follows once we combine the upper and lower bounds on $|F(x, y)|_p$ and use the inequality $|c_d|_p \geq c_d^{-1}$.
\end{proof}

\begin{lem} \label{lem:siegels_lemma}
\emph{(Siegel's lemma, \cite{bombieri-vaaler})} Let $N$ and $M$ be positive integers with \mbox{$N > M$}. Let $a_{i, j}$ be integers of absolute value at most $A \geq 1$ for $i = 1, \ldots, N$ and $j = 1, \ldots, M$. Then there exist integers $t_1, \ldots, t_N$, not all zero, such that
$$
|t_i| \leq (NA)^{\frac{M}{N - M}}, \quad\,\, \sum\limits_{i = 1}^N a_{i, j}t_i = 0, \quad\,\, j = 1, \ldots, M.
$$
\end{lem}

\begin{proof}
See, for example, \cite[Lemma 2.7]{zannier}.
\end{proof}

\begin{lem} \label{lem:bound_on_coefficients_2}
Let $\alpha$ be an algebraic number of degree $d$ over $\Q$. Then for every non-negative integer $r$ there exist rational numbers $a_{r, i}$ such that
$$
\alpha^r = a_{r, d-1}\alpha^{d-1} + \cdots + a_{r, 1}\alpha +  a_{r, 0}.
$$
Furthermore, if we denote the leading coefficient of the minimal polynomial of $\alpha$ by $c_\alpha$ and put
$$
C_8 = 1 + \max\limits_{0 \leq i \leq d - 1}\left\{|a_{d, i}|\right\},
$$
then $c_\alpha^{\max\{0, r - d + 1\}}a_{r, i} \in \Z$ and $|a_{r, i}| \leq C_8^{\max\{0, r - d + 1\}}$ for all $i$ such that $0 \leq i \leq d - 1$.
\end{lem}

\begin{proof}
See \cite[Proposition 2.6]{zannier}.
\end{proof}


Let $P \in \mathbb C[x]$ be a polynomial of degree $d \geq 1$. The \emph{house} of $P$, denoted $\house{P}$, is defined to be
$$
\house{P} = \max\left\{|\alpha_1|, \ldots, |\alpha_d|\right\},
$$
where $\alpha_1, \ldots, \alpha_d \in \mathbb C$ are the roots of $P$. For an algebraic number $\alpha$, we define the \emph{house} of $\alpha$ as $\house{\alpha} = \house{f}$, where $f$ is the minimal polynomial of $\alpha$.

Let $\alpha$ be an algebraic number and $\mathcal O$ the ring of integers of $\Q(\alpha)$. Let $c_\alpha$ denote the leading coefficient of the minimal polynomial of $\alpha$. We define
\begin{equation} \label{eq:theta}
\theta_\alpha = \left[\mathcal O \colon \Z[c_\alpha\alpha]\right].
\end{equation}
That is, $\theta_\alpha$ is equal to the index of $\Z[c_\alpha\alpha]$ in the additive group $\mathcal O$.

\begin{lem} \label{lem:alpha_beta_bound}
Let $\alpha$ be an algebraic number of degree $d$ over $\Q$ and
$$
\beta = b_{d - 1}\alpha^{d - 1} + \cdots + b_1\alpha + b_0,
$$
where $b_0, b_1, \ldots, b_{d - 1} \in \Q$. There exists a positive number $C_9$, which depends only on $\alpha$ and $\beta$, such that
$$
\max\limits_{1 \leq i \leq d}\{|b_i|\} \leq C_9.
$$
Furthermore,
$$
\theta_\alpha c_\beta\beta \in \Z[c_\alpha\alpha],
$$
where $\theta_\alpha$ is defined in (\ref{eq:theta}). In particular, $\theta_\alpha c_\beta b_i \in \Z$ for all \mbox{$i = 0, 1, \ldots, d - 1$}.
\end{lem}

\begin{proof}
Let $\alpha = \alpha_1, \ldots, \alpha_d$ denote the conjugates of $\alpha$. For $j = 1, \ldots, d$, let
$$
\beta_j = \sum\limits_{i = 0}^{d - 1}b_i\alpha_j^i.
$$
Then each $\beta_j$ is a conjugate of $\beta = \beta_1$. Further,
\begin{equation} \label{eq:vandermonde}
\begin{pmatrix}
\beta_1\\
\beta_2\\
\vdots\\
\beta_d
\end{pmatrix}
=
\begin{pmatrix}
1 & \alpha_1 & \alpha_1^2 & \ldots & \alpha_1^{d - 1}\\
1 & \alpha_2 & \alpha_2^2 & \ldots & \alpha_2^{d - 1}\\
\vdots & \vdots & \vdots & \ddots & \vdots\\
1 & \alpha_d & \alpha_d^2 & \ldots & \alpha_d^{d - 1}
\end{pmatrix}
\cdot
\begin{pmatrix}
b_0\\
b_1\\
\vdots\\
b_{d-1}
\end{pmatrix}.
\end{equation}
Let us denote the Vandermonde matrix on the right-hand side of the above expression \mbox{by $V$.} Then it follows from the inequality (4.1) in \cite{gautschi} that
$$
\|V^{-1}\|_\infty \leq \max\limits_{1 \leq j \leq d}\prod\limits_{\substack{1 \leq i \leq d\\ i \neq j}}\frac{1 + |\alpha_i|}{|\alpha_i - \alpha_j|},
$$
where $\|\,\cdot\,\|_\infty$ denotes the matrix infinity norm.

Now, let $V^{-1} = (v_{ij})$. Then it follows from (\ref{eq:vandermonde}) that $b_i = \sum_{j = 1}^d v_{ij}\beta_j$, so
$$
|b_i| \leq \sum\limits_{j = 1}^d|v_{ij}|\cdot |\beta_j| \leq d\cdot \house{\beta}\cdot  \max\limits_{1 \leq j \leq d}\prod\limits_{\substack{1 \leq i \leq d\\ i \neq j}}\frac{1 + |\alpha_i|}{|\alpha_i - \alpha_j|}.
$$

Next, let $c_\alpha$ and $c_\beta$ denote the leading coefficients of the minimal polynomials of $\alpha$ and $\beta$, respectively. Note that $\theta_\alpha c_\beta \beta = (\#\mathcal O/\Z[c_\alpha\alpha])c_\beta \beta \in \Z[c_\alpha \alpha]$ due to the fact that $c_\beta\beta \in \mathcal O$. Finally, observe that
$$
\theta_\alpha c_\beta \beta = \sum\limits_{i = 0}^{d-1}\theta_\alpha c_\beta b_i\alpha^i 
\in \Z[c_\alpha\alpha]
$$
Since $\Z[c_\alpha\alpha] \subseteq \Z[\alpha]$, it must be the case that each coefficient $\theta_\alpha c_\beta b_i$ is an integer.
\end{proof}

%

Let $P \in \mathbb C[x]$ be a polynomial that is not identically equal to zero, with leading coefficient $c_P$. The \emph{Mahler measure} of $P$, denoted $M(P)$, is defined to be $M(P) = |c_P|$ if $P$ is the constant polynomial and
$$
M(P) = |c_P|\prod\limits_{i = 1}^d\max\{1, |\alpha_i|\}
$$
otherwise, where $\alpha_1, \ldots, \alpha_d \in \mathbb C$ are the roots of $P$. For a binary form $Q \in \mathbb C[x, y]$, we define the Mahler measure of $Q$ as $M(Q) = M(Q(x, 1))$. Finally, for an algebraic number $\alpha$, we define the Mahler measure of $\alpha$ to be $M(\alpha) = M(f)$, where $f$ is the minimal polynomial of $\alpha$.

The following lemma is a reformulation of a well-known result of Lewis and Mahler \cite{lewis-mahler}.

\begin{lem} \label{lem:lewis-mahler}
Let
$$
F(x, y) = c_dx^d + c_{d-1}x^{d-1}y + \cdots + c_0y^d
$$
be a binary form of degree $d \geq 2$ with integer coefficients such that $c_0c_d \neq 0$. Let $x_1$ and $y_1$ be nonzero integers. There exists a root $\alpha$ of $F(x, 1)$ such that
$$
\min\left\{\left|\alpha - \frac{x_1}{y_1}\right|, \left|\alpha^{-1} - \frac{y_1}{x_1}\right|\right\} \leq \frac{C_{10}|F(x_1, y_1)|}{H(x_1, y_1)^d},
$$
where
$$
C_{10} =  \frac{2^{d-1}d^{(d-1)/2}M(F)^{d-2}}{|D(F)|^{1/2}}.
$$
\end{lem}

\begin{proof}
Let $\alpha$ be a root of $F(x, 1)$ that minimizes $|\alpha - x/y|$. By \cite[Lemma 3]{stewart2},
$$
\left|\alpha - \frac{x_1}{y_1}\right| \leq \frac{C_{10}|F(x_1, y_1)|}{|y_1|^d}.
$$
If $|y_1| \geq |x_1|$, then $H(x_1, y_1) = |y_1|$, and so the result holds. Otherwise, since $c_0c_d \neq 0$, we see that the roots of $F(x, 1)$ and $F(1, x)$ are nonzero, meaning that all roots of $F(1, x)$ are of the form $\alpha^{-1}$, where $\alpha$ is a root of $F(x, 1)$. If we let $\beta^{-1}$ be a root of $F(1, x)$ that minimizes $|\beta^{-1} - y_1/x_1|$, then it follows from \mbox{\cite[Lemma 3]{stewart2}} that
$$
\left|\beta^{-1} - \frac{y_1}{x_1}\right| \leq \frac{C_{10}|F(x_1, y_1)|}{|x_1|^d}.
$$
Since $|x_1| > |y_1|$, the result follows.
\end{proof}

\begin{lem} \label{lem:unique}
Let $K = \R$ or $\Q_p$, where $p$ is a rational prime, and let $\overline K$ denote the algebraic closure of $K$. Denote the standard absolute value on $K$ by $|\quad|$. Let $\alpha$ and $\beta$ be distinct numbers in $\overline{K}$. Let $\mu$ and $C_0$ be positive real numbers.

If $x_1/y_1$ is a rational number such that $H(x_1, y_1) \geq (2C_0/|\alpha - \beta|)^{1/\mu}$ and
$$
\left|\alpha - \frac{x_1}{y_1}\right| < \frac{C_0}{H(x_1, y_1)^\mu},
$$
then
$$
\left|\beta - \frac{x_1}{y_1}\right| \geq \frac{C_0}{H(x_1, y_1)^\mu}.
$$
\end{lem}

\begin{proof}
Suppose that the statement is false. Then it follows from the triangle inequality that
$$
|\alpha - \beta| \leq \left|\alpha - \frac{x_1}{y_1}\right| + \left|\beta - \frac{x_1}{y_1}\right| < \frac{2C_0}{H(x_1, y_1)^\mu},
$$
and so $H(x_1, y_1) < (2C_0/|\alpha - \beta|)^{1/\mu}$, leading us to a contradiction.
\end{proof}

\begin{cor} \label{cor:unique}
Let $K = \R$ or $\Q_p$, where $p$ is a rational prime, and let $\overline K$ denote the algebraic closure of $K$. Denote the standard absolute value on $K$ by $|\quad|$. Let $f(x) \in \Z[x]$ be an irreducible polynomial of degree $d \geq 2$ with roots $\alpha_1, \ldots, \alpha_d \in \overline{K}$. Let $\mu$  and $C_0$ be positive real numbers. There exists a positive number $C_{11}$, which depends only on $f$, $\mu$ and $C_0$, with the following property. If $x_1/y_1$ is a rational number such that $H(x_1, y_1) \geq C_{11}$ and
$$
\left|\alpha_i - \frac{x_1}{y_1}\right| < \frac{C_0}{H(x_1, y_1)^\mu}
$$
for some $i \in \{1, \ldots, d\}$, then
$$
\left|\alpha_j - \frac{x_1}{y_1}\right| \geq \frac{C_0}{H(x_1, y_1)^\mu}
$$
for all $j \neq i$.
\end{cor}

\begin{lem} \label{lem:thue-siegel}
\emph{(The Thue-Siegel Principle \cite{bombieri-mueller})} Let $K = \R$ or $\Q_p$, where $p$ is a rational prime, and denote the standard absolute value on $K$ by $|\quad|$. Let $\alpha_1 \in K$ be an algebraic number of degree $d \geq 3$ over $\Q$ and let $\alpha_2 \in \Q(\alpha_1)$ have degree $d$. Let $t$ and $\tau$ be such that
\begin{equation} \label{eq:thue-siegel_intervals}
\frac{2 + \sqrt{2d^3 + 2d^2 - 4d}}{d(d + 1)} < t < \sqrt{\frac{2}{d}}, \quad \sqrt{2 - dt^2} < \tau < t - \frac{2}{d},
\end{equation}
and put $\lambda = 2/(t - \tau)$, so that $\lambda < d$. Define
$$
A_i = \frac{t^2}{2 - dt^2}\left(\log M(\alpha_i) + \frac{d}{2}\right) \quad \text{for $i = 1, 2$}.
$$
Let $x_1/y_1$ and $x_2/y_2$ be rational numbers in lowest terms that satisfy the inequalities
$$
\left|\alpha_i - \frac{x_i}{y_i}\right| < \frac{1}{\left(4e^{A_i}H(x_i, y_i)\right)^\lambda} \quad \text{for $i = 1, 2$.}
$$
Then
$$
\log(4e^{A_2}) + \log H(x_2, y_2) \leq \delta^{-1}\left(\log(4e^{A_1}) + \log H(x_1, y_1)\right),
$$
where
$$
\delta = \frac{dt^2 + \tau^2 - 2}{d - 1}.
$$
\end{lem}

\begin{proof}
Since $d \geq 3$, the intervals in (\ref{eq:thue-siegel_intervals}) are guaranteed to be non-empty, so the statement is not vacuously true. Since
$$
\left|\alpha_i - \frac{x_i}{y_i}\right| < \frac{1}{\left(4e^{A_i}H(x_i, y_i)\right)^\lambda} < \frac{1}{\left(3e^{A_i}H(x_i, y_i)\right)^\lambda} \quad \text{for $i = 1, 2$},
$$
the case when $|\alpha_1| \leq 1$ and $|\alpha_2| \leq 1$ follows directly from \mbox{\cite[Section II]{bombieri-mueller}}. More precisely, the comment on \mbox{p. 184} of \cite{bombieri-mueller} states that the triple $(A_1, A_2, \tau)$ is admissible for the data $(\alpha_1, \alpha_2, x_1/y_1, x_2/y_2, t, \vartheta, \delta)$, where we take $\vartheta = 1$. Note also that the comments on p. 74 of \cite{bombieri-schmidt} apply in our situation:
\begin{enumerate}[(i)]
\item the hypothesis $K_{\tilde v} = k_v$ is not used in the proof and therefore may be omitted;
\item $c(\vartheta t) \leq \log 3$;
\item the chosen value for $A_i$ implies a fortiori $\left|\alpha_i - \frac{x_i}{y_i}\right| < \vartheta(t - \tau)$ for $i = 1, 2$;
\item $h(x_i/y_i) = H(x_i, y_i)$ for $i = 1, 2$;
\item the exponent in (5A), p. 179 of \cite{bombieri-mueller} should be $2\vartheta/(t - \tau)$, not $2\vartheta^{-1}/(t - \tau)$.
\end{enumerate}

Next, we consider the case when $|\alpha_i| > 1$ for some $i \in \{1, 2\}$. If $K = \mathbb R$, then $|x_i/y_i| > 1$, and so
$$
\left|\alpha_i^{-1} - \frac{y_i}{x_i}\right| < |\alpha_i|^{-1}|y_i/x_i| \left(4e^{A_i}H(x_i, y_i)\right)^{-\lambda} < \left(3e^{A_i}H(x_i, y_i)\right)^{-\lambda}.
$$
If $K = \mathbb Q_p$, then $|y_i| \leq 1$ and we claim that $|x_i| = 1$. For suppose not and $|x_i| < 1$. Since $x_i/y_i$ is in lowest terms it must be the case that $|y_i| = 1$, so
$$
|x_i\alpha_i^{-1}| = |x_i| \cdot |\alpha_i^{-1}| < 1 = |y_i|.
$$
Since $|x_i\alpha_i^{-1}| \neq |y_i|$, it follows from the strong triangle inequality that
$$
|x_i\alpha_i^{-1} - y_i| = \max\left\{|x_i\alpha_i^{-1}|,\ |y_i|\right\} = \max\left\{|x_i\alpha_i^{-1}|,\ 1\right\} \geq 1.
$$
Thus,
$$
1 \leq |x_i\alpha_i^{-1} - y_i| = |y_i| |\alpha_i^{-1}| \left|\alpha_i - \frac{x_i}{y_i}\right| < \left(4e^{A_i}H(x_i, y_i)\right)^{-\lambda},
$$
which is impossible. Hence $|x_i| = 1$, so
$$
\left|\alpha_i^{-1} - \frac{y_i}{x_i}\right| < |\alpha_i|^{-1}|y_i/x_i| \left(4e^{A_i}H(x_i, y_i)\right)^{-\lambda} < \left(3e^{A_i}H(x_i, y_i)\right)^{-\lambda}.
$$
We conclude that, as long as $\left|\alpha_i - \frac{x_i}{y_i}\right| < \left(4e^{A_i}H(x_i, y_i)\right)^{-\lambda}$ for $i = 1, 2$, the inequalities
$$
\left|\alpha_i - \frac{x_i}{y_i}\right| < \frac{1}{\left(3e^{A_i}H(x_i, y_i)\right)^{\lambda}} \quad \text{and} \quad \left|\alpha_i^{-1} - \frac{y_i}{x_i}\right| < \frac{1}{\left(3e^{A_i}H(x_i, y_i)\right)^{\lambda}}
$$
hold whenever $|\alpha_i| > 1$ for some $i \in \{1, 2\}$. Consequently, we can always choose $r, s \in \{-1, 1\}$ so that $|\alpha_1^r| \leq 1$, $|\alpha_2^s| \leq 1$ and $(A_1, A_2, \tau)$ is admissible for the data $\left(\alpha_1^r, \alpha_2^s, (x_1/y_1)^r, (x_2/y_2)^s, t, \vartheta, \delta\right)$. The result now follows from \mbox{\cite[Section II]{bombieri-mueller}}.
\end{proof}

\section{Minimal Pairs} \label{sec:minimal_pairs}

Let $\alpha$ be an algebraic number of degree $d$ over $\Q$ and let $\beta \in \Q(\alpha)$ be irrational. With a pair $(\alpha, \beta)$ we associate two polynomials $P, Q \in \Z[x]$, which possess certain minimal properties listed in Definition \ref{defin:minimal_pair}. The properties of minimal pairs summarized in Proposition \ref{prop:properties_of_minimal_pairs} will play a crucial role in proofs of Archimedean and non-Archimedean gap principles, which are outlined in Sections \ref{sec:generalized_archimedean_gap_principle} and \ref{sec:generalized_non-archimedean_gap_principle}, respectively.

\begin{defin} \label{defin:minimal_pair}
Let $\alpha$ be an algebraic number of degree $d$ and let $\beta \in \Q(\alpha)$ be irrational. We say that two univariate polynomials, $P$ and $Q$, not both identically equal to zero, form a \emph{minimal pair} for $(\alpha, \beta)$, if they satisfy the following four properties:
\begin{enumerate}[(1)]
\item $P, Q \in \Z[x]$.

\item $P(\alpha) + \beta Q(\alpha) = 0$.

\item The quantity $\max\{\deg P, \deg Q\}$ is minimal among all polynomials satisfying properties (1) and (2).

\item The quantity $\max\{H(P), H(Q)\}$ is minimal among all polynomials satisfying properties (1), (2) and (3).
\end{enumerate}
If $P, Q$ is a minimal pair for $(\alpha, \beta)$, we write
$$
r(\alpha, \beta) = \max\{\deg P, \deg Q\}.
$$
\end{defin}

If $P, Q$ is a minimal pair for $(\alpha, \beta)$ then $-P, -Q$ is also a minimal pair for $(\alpha, \beta)$. This already demonstrates that minimal pairs are not unique. Furthermore, the uniqueness is not guaranteed even if we impose an additional condition that the leading coefficient of $Q$ is equal to one. Indeed, let
$$
\alpha = 2\cos\left(\frac{2\pi}{15}\right) \quad \text{and} \quad \beta = 2\cos\left(\frac{4\pi}{15}\right).
$$
Then both
$$
P_1(x) = -x^2 + 2, \quad Q_1(x) = 1
$$
and
$$
P_2(x) = -x^2 + 2x - 1, \quad Q_2(x) = x^2 - x - 1
$$
are minimal pairs for $(\alpha, \beta)$.

If $P, Q$ is a minimal pair for $(\alpha, \beta)$, then we can define a polynomial
$$
R(x, y) = P(x) + yQ(x).
$$
Polynomials of such form were used by Thue \cite{thue} for the purpose of establishing the first instance of the Thue-Siegel principle \cite{bombieri-mueller}. More precisely, they were constructed as to achieve high vanishing at the point $(\alpha, \alpha)$, i.e., \mbox{$D_iR(\alpha, \alpha) = 0$} for \mbox{$i = 0, 1, \ldots, \ell$} for some large $\ell$ (see the exposition of Thue's method in \mbox{\cite[Chapter 2]{zannier}}). In turn, we construct $R(x, y)$ so to achieve $R(\alpha, \beta) = 0$ for arbitrary irrational $\beta \in \Q(\alpha)$ for the purpose of obtaining a generalized gap principle. The following proposition summarizes various properties of minimal pairs.

\begin{prop} \label{prop:properties_of_minimal_pairs}
Let $\alpha$ be an algebraic number of degree $d$ over $\Q$ and let \mbox{$\beta \in \Q(\alpha)$} be irrational. Let $P, Q$ be a minimal pair for $(\alpha, \beta)$ and put $r = r(\alpha, \beta)$. Then the polynomials $P$, $Q$, and their Wronskian $W = PQ' - QP'$ possess the following properties.
\begin{enumerate}
\item
\begin{equation} \label{eq:r_bound}
1 \leq r \leq \left\lfloor d/2\right\rfloor.
\end{equation}

\item $P$ and $Q$ are coprime.

\item If $\hat P, \hat Q \in \Z[x]$ satisfy $\hat P(\alpha) + \beta \hat Q(\alpha) = 0$ and $\max\{\deg \hat P, \deg \hat Q\} \leq d - 1 - r$, then $\hat P = GP$, $\hat Q = GQ$ for some $G \in \Z[x]$.

\item There exists a positive number $C_{12}$, which depends only on $\alpha$ and $\beta$, such that
\begin{equation} \label{eq:height_PQ}
\max\{H(P), H(Q)\} \leq C_{12}.
\end{equation}

%
%

\item If $\alpha \in \CC$, there exists a positive real number $C_{13}$, which depends only on $\alpha$ and $\beta$, such that
\begin{equation} \label{eq:W_alpha_archimedean_lower_bound}
|W(\alpha)| \geq C_{13}.
\end{equation}
Similarly, if $\alpha \in \Q_p$ for some rational prime $p$, there exists a positive real number $C_{14}$, which depends only on $\alpha$ and $\beta$, such that
\begin{equation} \label{eq:W_alpha_non-archimedean_lower_bound}
|W(\alpha)|_p \geq C_{14}.
\end{equation}
%
\end{enumerate}
\end{prop}

\begin{proof}
Let us prove each of the above statements.

\begin{enumerate}
\item First, we prove that $r \geq 1$. If not, then $r = \max\{\deg P, \deg Q\} = 0$, which means that $P = p$ and $Q = q$ for some integers $p$ and $q$, not both equal to zero. If $q \neq 0$, then $P(\alpha) + \beta Q(\alpha) = 0$ implies $\beta = -p/q$, which contradicts the fact that $\beta$ is irrational. If $q = 0$, then $p = 0$, which is impossible, since we assumed that both $p$ and $q$ cannot be equal to zero. Thus, $r \geq 1$.

Next, we prove that $r \leq s$, where $s = \lfloor d/2\rfloor$. Write
\begin{equation} \label{eq:P_hat_Q_hat}
\hat P(x) = \sum\limits_{i = 0}^s a_ix^i \quad \text{and} \quad \hat Q(x) = \sum\limits_{i = 0}^s a_{s + 1 + i}x^i.
\end{equation}
We view the $2s + 2$ integer coefficients $a_0, \ldots, a_{2s + 1}$ as variables. Since $\alpha$ is algebraic of degree $d$ over $\Q$ and $\beta \in \Q(\alpha)$, the equation $\hat P(\alpha) + \beta \hat Q(\alpha) = 0$ defines $d$ linear equations over $\Q$, which we will define in the proof of Part 4. Since $2s + 2 > d$, the existence of a non-trivial integer solution to the system of $d$ linear equations over $\Q$ in $2s + 2$ variables is guaranteed by Lemma \ref{lem:siegels_lemma}. Therefore, there exist polynomials $\hat P, \hat Q$, not both zero, such that $\max\{\deg \hat P, \deg \hat Q\} \leq s$. Consequently, the polynomials $P, Q$ with $\max\{\deg P, \deg Q\}$ minimal satisfy
$$
\max\{\deg P, \deg Q\} \leq \max\{\deg \hat P, \deg \hat Q\} \leq s.
$$

\item Let $G = \gcd(P, Q)$ and suppose that $\deg G \geq 1$. Then certainly $G(\alpha) \neq 0$, because $\alpha$ has degree $d$ and $\deg G \leq \deg P < d$. Put $\hat P = P/G$ and $\hat Q = Q/G$. Then
$$
\hat P(\alpha) + \beta\hat Q(\alpha) = 0
$$
and
$$
\max\{\deg \hat P, \deg \hat Q\} < \max\{\deg P, \deg Q\},
$$
in contradiction to our assumption that $\max\{\deg P, \deg Q\}$ is minimal. This means that $\deg G = 0$, and so $P$ and $Q$ are coprime.

\item Since
$$
P(\alpha) + \beta Q(\alpha) = \hat P(\alpha) + \beta \hat Q(\alpha) = 0,
$$
we have
$$
P(\alpha)\hat Q(\alpha) - Q(\alpha)\hat P(\alpha) = 0.
$$
Since $\alpha$ has degree $d$ and
\begin{align*}
\deg\left(P\hat Q - Q\hat P\right)
& \leq \max\{\deg P, \deg Q\} + \max\{\deg \hat P, \deg \hat Q\}\\
& \leq r + (d - 1 - r)\\
& < d,
\end{align*}
we conclude that $P\hat Q - Q\hat P$ is identically equal to zero. If $\hat Q = 0$, then $\hat P = 0$, and so $G = 0$. Otherwise $P/Q = \hat P/\hat Q$. If we put $G = \gcd(\hat P, \hat Q)$, then it becomes clear that $\hat P = GP$, $\hat Q = GQ$.

\item Define $b_i, c_{k, i} \in \Q$ as follows:
$$
\alpha^k = c_{k, d - 1}\alpha^{d - 1} + \cdots + c_{k, 1}\alpha + c_{k, 0},
$$
$$
\beta = b_{d - 1}\alpha^{d - 1} + \cdots + b_1\alpha + b_0.
$$ 
Let $\hat P, \hat Q$ be as in (\ref{eq:P_hat_Q_hat}). Then,
\begin{align*}
\hat P(\alpha) + \beta\hat Q(\alpha)
& = \sum\limits_{i = 0}^s a_i\alpha^i + \left(\sum\limits_{i = 0}^{d - 1}b_i\alpha^i\right)\cdot\left(\sum\limits_{j = 0}^s a_{s + 1 + j}\alpha^j\right)\\
& = \sum\limits_{i = 0}^sa_i\alpha^i + \sum\limits_{j = 0}^sa_{s + 1 + j}\sum\limits_{i = 0}^{d - 1}b_i\alpha^{i + j}\\
& = \sum\limits_{i = 0}^sa_i\alpha^i + \sum\limits_{j = 0}^sa_{s + 1 + j}\left(\sum\limits_{i = j}^{d - 1}b_{i-j}\alpha^i + \sum\limits_{k = d}^{d - 1 + j}b_{k - j}\alpha^k\right)\\
& = \sum\limits_{i = 0}^sa_i\alpha^i + \sum\limits_{j = 0}^sa_{s + 1 + j}\sum\limits_{i = j}^{d - 1}b_{i-j}\alpha^i + \sum\limits_{j = 0}^sa_{s + 1 + j}\sum\limits_{k = d}^{d - 1 + j}b_{k - j}\sum\limits_{i = 0}^{d - 1}c_{k, i}\alpha^i\\
& = \sum\limits_{i = 0}^s\left(a_i + \sum\limits_{j = 0}^ib_{i - j}a_{s + 1 + j} + \sum\limits_{j = 0}^s\sum\limits_{k = d}^{d - 1 + j}b_{k - j}c_{k, i}a_{s + 1 + j}\right)\alpha^i +\\
& + \sum\limits_{i = s + 1}^{d - 1}\sum\limits_{j = 0}^s\left(b_{i-j} + \sum\limits_{k = d}^{d - 1 +  j}b_{k - j}c_{k, i}\right)a_{s + 1 + j}\alpha^i\\
& = \sum\limits_{i = 0}^{d - 1}L_i(\vec a)\alpha^i,
\end{align*}
where $\vec a = (a_0, a_1, \ldots, a_{2s + 1})$ and
\small
$$
L_i(\vec a) =
\begin{cases}
a_i + \sum\limits_{j = 0}^i\left(b_{i - j} + \sum\limits_{k = d}^{d - 1 + j}b_{k - j}c_{k, i}\right)a_{s + 1 + j} + \sum\limits_{j = i + 1}^s\left(\sum\limits_{k = d}^{d - 1 + j}b_{k - j}c_{k, i}\right)a_{s + 1 + j} & \text{if $i \leq s$,}\\
\sum\limits_{j = 0}^s\left(b_{i-j} + \sum\limits_{k = d}^{d - 1 +  j}b_{k - j}c_{k, i}\right)a_{s + 1 + j} & \text{if $i \geq s + 1$.}
\end{cases}
$$
\normalsize
We conclude that the equation $\hat P(\alpha) + \beta \hat Q(\alpha) = 0$ is equivalent to the system of $d$ linear equations $L_0(\vec a) = \ldots = L_{d - 1}(\vec a) = 0$ over $\Q$.

Put
$$
B = \max\limits_{0 \leq i \leq d - 1}\{|b_i|\} \quad \text{and} \quad C = \max\limits_{\substack{0 \leq k \leq d - 1 + s\\0 \leq i \leq d - 1}}\{|c_{k, i}|\}.
$$
Then we can bound the (rational) coefficients of $L_i(\vec a)$ from above by $B(1 + sC)$:
$$
\left|b_{i - j} + \sum\limits_{k = d}^{d - 1 + j}b_{k - j}c_{k, i}\right| \leq B + jBC \leq B(1 + sC),
$$
$$
\left|\sum\limits_{k = d}^{d - 1 + j}b_{k - j}c_{k, i}\right| \leq jBC \leq sBC < B(1 + sC).
$$

By Lemma \ref{lem:alpha_beta_bound} we have $\theta_\alpha c_\beta b_i \in \Z$ for all $i$ and $B \leq C_9$. Further, by Lemma \ref{lem:bound_on_coefficients_2} we have $c_\alpha^{\max\{0, k - d + 1\}}c_{k, i} \in \Z$ for all $i, k$ and $C \leq C_8^s$.
%
%
Hence the linear forms
$$
\hat L_i(\vec a) = \theta_\alpha c_\beta c_\alpha^sL_i(\vec a)
$$
have integer coefficients and the size of these coefficients is at most
$$
A = \theta_\alpha c_\beta c_\alpha^sC_9(1 + sC_8^s).
$$
By Lemma \ref{lem:siegels_lemma},
\small
\begin{align*}
\max\{H(\hat P), H(\hat Q)\}
& = \max\limits_{0 \leq i \leq 2s + 1}\{|a_i|\}\\
& \leq \left((2s + 2)A\right)^{d/(2s + 2 - d)}\\
& \leq \left((2s + 2)\theta_\alpha c_\beta c_\alpha^sC_9(1 + sC_8^s)\right)^{d/(2s + 2 - d)}.
\end{align*}
\normalsize

Now that we know an upper bound on $\max\{H(\hat P), H(\hat Q)\}$, we can determine an upper bound on $\max\{H(P), H(Q)\}$ by considering the following two cases.

\begin{itemize}
\item[\textbf{Case 1.}] Suppose that $\max\{\deg \hat P, \deg \hat Q\} > d - 1 - r$. Then it follows from Part 1 and the inequality $\max\{\deg \hat P, \deg  \hat Q\} \leq \lfloor d/2\rfloor$ that
$$
d \leq r + \max\{\deg \hat P, \deg \hat Q\} \leq 2\lfloor d/2\rfloor.
$$
Thus, $d$ is even and $\max\{\deg \hat P, \deg \hat Q\} = r = d/2$. Therefore the pair $\hat P, \hat Q$ satisfies Properties (1), (2), (3) in Definition \ref{defin:minimal_pair}. By Property (4), the polynomials $P$ and $Q$ satisfy
$$
\max\{H(P), H(Q)\} \leq \max\{H(\hat P), H(\hat Q)\},
$$
and so the result follows.

\item[\textbf{Case 2.}] Suppose that $\max\{\deg \hat P, \deg \hat Q\} \leq d - 1 - r$. Then we can use \mbox{Part 3} to conclude that $\hat P = GP$, $\hat Q = GQ$ for some $G \in \Z[x]$. Since either $\hat P$ or $\hat Q$ is nonzero, we have $H(G) \geq 1$. By Gelfond's Lemma \mbox{\cite[Lemma 1.6.11]{bombieri-gubler}},
$$
H(P) \leq H(G)H(P) \leq 2^{\deg(GP)}H(GP) \leq 2^{d/2}H(\hat P).
$$
An analogous estimate for $H(Q)$ yields the result.
\end{itemize}

\item Since $P$ and $Q$ are coprime and $r \geq 1$, they are linearly independent \mbox{over $\Q$}, so the Wronskian $W = PQ' - QP'$ is not identically equal to zero. Since $\alpha$ has degree $d$ and
\begin{align*}
\deg W
& = \deg(PQ' - QP')\\
& \leq \max\{\deg P, \deg Q\} + \max\{\deg P', \deg Q'\}\\
& \leq d/2 + (d/2 - 1)\\
& < d,
\end{align*}
we conclude that $W(\alpha) \neq 0$.

With the basic properties of heights listed in \cite[Section 2.4.1]{zannier}, we find the following upper bound on $H(W)$:
\begin{align*}
H(W)
& \leq H(PQ') + H(QP')\\
& \leq r(H(P)H(Q') + H(Q)H(P'))\\
& \leq 2r^2H(P)H(Q)\\
& \leq 2r^2\max\{H(P), H(Q)\}^2\\
& \leq 2(d/2)^2C_{12}^2.
\end{align*}

Suppose that $\alpha \in \CC$. Then $c_\alpha^{\deg W}W(\alpha)$ is a nonzero algebraic integer, so
$$
N_{\Q(\alpha)/\Q}\left(c_\alpha^{\deg W}W(\alpha)\right) = c_\alpha^{d\deg W}\prod_{i = 1}^dW(\alpha_i)
$$
is a nonzero rational integer. Thus,
\begin{align*}
|W(\alpha)|^{-1}
& \leq c_\alpha^{d\deg W}\prod\limits_{i = 2}^d|W(\alpha_i)|\\
& \leq c_\alpha^{d\deg W}\prod\limits_{i = 2}^d(\deg W + 1)H(W)\max\{1, |\alpha_i|\}^{\deg W}\\
& \leq (2rH(W))^{d - 1}\left(\frac{c_\alpha^{d - 1}M(\alpha)}{\max\{1, |\alpha|\}}\right)^{2r - 1}\\
& \leq \left((d^3/2)C_{12}^2\right)^{d - 1}\left(\frac{c_\alpha^{d - 1}M(\alpha)}{\max\{1, |\alpha|\}}\right)^{d - 1}.
\end{align*}
\end{enumerate}

Suppose that $\alpha \in \Q_p$. Let $f(x)$ denote the minimal polynomial of $\alpha$ with the leading coefficient $c_\alpha$. By \cite[Theorem 1.3.2]{prasolov}, there exist polynomials $\varphi, \psi \in \Z[x]$ such that $\deg \varphi < \deg W$, $\deg \psi < d$, and
$$
\varphi(x)f(x) + \psi(x)W(x) = \Res(f, W).
$$
Here $\Res(f, W)$ denotes the resultant of $f$ and $W$. Since $\Res(f, W) \neq 0$ and $\alpha$ is a root of $f(x)$, we see that $\psi(\alpha)W(\alpha) = \Res(f, W)$. Since $c_\alpha^{d - 1}\psi(\alpha)$ is an algebraic integer, its $p$-adic absolute value does not exceed one, so
$$
|W(\alpha)|_p \geq |c_\alpha^{d-1}\psi(\alpha)W(\alpha)|_p = |c_\alpha^{d-1}\Res(f, W)|_p.
$$
Further, it follows from Hadamard's inequality, as well as the upper bound on $H(W)$ established previously, that
\begin{align*}
|\Res(f, W)|
& \leq (\deg f + 1)^{\deg W/2}(\deg W + 1)^{\deg f/2}H(\alpha)^{\deg W}H(W)^{\deg f}\\
& \leq (d + 1)^{(2r - 1)/2}(2r)^{d/2}H(\alpha)^{2r - 1}(2(d/2)^2C_{12}^2)^{d}\\
& \leq (d + 1)^{(d - 1)/2}d^{d/2}H(\alpha)^{d - 1}((d^2/2)C_{12}^2)^{d}.
\end{align*}
Combining the lower bound on $|W(\alpha)|_p$ with an upper bound on $|\Res(f, W)|$ yields the result:
\begin{align*}
|W(\alpha)|_p
& \geq |c_\alpha^{d-1}\Res(f, W)|_p\\
& \geq |c_\alpha^{d-1}\Res(f, W)|^{-1}\\
& \geq H(\alpha)^{-(d - 1)}|\Res(f, W)|^{-1}\\
& \geq (d + 1)^{-(d - 1)/2}d^{-d/2}H(\alpha)^{-2d+2}((d^2/2)C_{12}^2)^{-d}.
\end{align*}
\end{proof}

\section{A Gap Principle in the Presence of Vanishing} \label{sec:gap-principle-with-vanishing}

Let $\alpha$ be an algebraic number over $\Q$ of degree $d \geq 2$ and let $\beta \in \Q(\alpha)$ be irrational. Let $P, Q \in \Z[x]$ be polynomials such that
$$
P(\alpha) + \beta Q(\alpha) = 0.
$$
In this section we prove Proposition \ref{prop:alternative_gap_principle}, which states that \emph{despite the vanishing} of $P(x) + yQ(x)$ at the point $\left(\frac{x_1}{y_1}, \frac{x_2}{y_2}\right) \in \Q^2$, it is still possible to produce a gap principle, provided that the quantity $r = \max\{\deg P, \deg Q\}$ exceeds one.

\begin{prop} \label{prop:alternative_gap_principle}
Let $P, Q \in \Z[x]$ be coprime and such that
$$
r = \max\{\deg P, \deg Q\} \geq 1.
$$
Let $x_1/y_1, x_2/y_2$ be rational numbers in lowest terms such that $H(x_2, y_2) \geq H(x_1, y_1)$ and
\begin{equation} \label{eq:R_vanishes_at_rational_point}
P\left(\frac{x_1}{y_1}\right) + \frac{x_2}{y_2}Q\left(\frac{x_1}{y_1}\right) = 0.
\end{equation}
Then
$$
H(x_2, y_2) \geq \frac{H(x_1, y_1)^r}{C_{15}\max\{H(P), H(Q)\}^{2r^2 + 3r}},
$$
where
$$
C_{15} = 2^{r^2} (r + 1)^{(3r^2 + 2r)/2}.
$$
\end{prop}

The proof of Proposition \ref{prop:alternative_gap_principle} is given at the end of the section, and it follows directly from the results established below.

\begin{lem} \label{lem:resultant}
Let $P, Q \in \Z[x]$ be coprime polynomials of degrees $r$ and $s$, respectively, such that $r \geq \max\{1, s\}$. Let $c_P$ be the leading coefficient of $P$,
$$
P(x, y) = y^rP(x/y) \quad \text{and} \quad Q(x, y) = y^rQ(x/y).
$$
Then for all coprime integers $a$ and $b$ the number $g = \gcd\left(P(a, b), Q(a, b)\right)$ divides
$$
\varrho = \left|c_P^{r-s}\Res(P, Q)\right|,
$$
where $\Res(P, Q)$ denotes the resultant of $P$ and $Q$. Furthermore,
$$
1 \leq \varrho \leq (r + 1)^r\max\{H(P), H(Q)\}^{2r}.
$$
\end{lem}

\begin{proof}
Let $a$ and $b$ be coprime integers and suppose that a prime power $p^n$ exactly divides $g = \gcd\left(P(a, b), Q(a, b)\right)$. Since $a$ and $b$ are coprime, either $a$ or $b$ is not divisible by $p$. Suppose that $p$ does not divide $b$. By \cite[Theorem 1.3.2]{prasolov}, there exist polynomials $\varphi, \psi \in \Z[x]$ such that
$$
\varphi(x)P(x) + \psi(x)Q(x) = \Res(P, Q).
$$
Let $t = \max\{\deg \varphi, \deg \psi\}$. We evaluate the polynomial on the left-hand side at $x = a/b$ and multiply both sides of the above equality by $b^{r + t}$:
$$
b^t\varphi(a/b)P(a, b) + b^t\psi(a/b)Q(a, b) = \Res(P, Q)b^{r+t}
$$
By definition of $t$, the numbers $b^t\varphi(a/b)$ and $b^t\psi(a/b)$ are integers. Since $p$ does not divide $b$ and $p^n$ divides both $P(a, b)$ and $Q(a, b)$, we conclude that $p^n$ divides $\Res(P, Q)$.

Suppose that $p$ divides $b$. Then $p$ does not divide $a$, and so by analogy with the previous case we see that $p^n$ divides $\Res(P(1, x), Q(1, x))$. Let $\mathcal R(f) = x^{\deg f}f(1/x)$ denote the reciprocal of a polynomial $f(x)$. Then
$$
P(1, x) = \mathcal R(P) \quad \text{and} \quad Q(1, x) = x^{r - s}\mathcal R(Q),
$$
so
\begin{align*}
\Res(P(1, x), Q(1, x))
& = \Res\left(\mathcal R(P), x^{r - s}\mathcal R(Q)\right)\\
& = \Res(\mathcal R(P), x)^{r - s}\Res(\mathcal R(P), \mathcal R(Q))\\
& = \left((-1)^rc_P\right)^{r - s}(-1)^{rs}\Res(P, Q)\\
& = (-1)^rc_P^{r - s}\Res(P, Q).
\end{align*}
Therefore, $p^n$ divides $\left|c_P^{r - s}\Res(P, Q)\right|$, and the result follows.

Finally, since $P(x)$ and $Q(x)$ are coprime and $r \geq 1$, we have $\Res(P, Q) \neq 0$, so $\varrho \geq 1$. Applying Hadamard's inequality and $r \geq s$, we obtain
\begin{align*}
|c_P^{r - s}\Res(P, Q)|
& \leq |c_P|^{r - s}(r + 1)^{s/2}(s + 1)^{r/2}H(P)^sH(Q)^r\\
& \leq (r + 1)^r\max\{H(P), H(Q)\}^{2r}.
\end{align*}
\end{proof}

\begin{lem} \label{lem:two_forms_bound}
Let
$$
P(x, y) = \prod\limits_{i = 1}^r(\alpha_ix + \beta_iy) \quad \text{and} \quad Q(x, y) = \prod\limits_{j = 1}^r(\gamma_jx + \delta_jy)
$$
be binary forms of degree $r \geq 1$, with complex coefficients. Let
\begin{equation} \label{eq:two_forms_bound_C}
C = C(P, Q) = \frac{\min_{i, j}\{|\alpha_i\delta_j -\beta_i\gamma_j|\}}{\max_{i, j}\left\{\max\{|\alpha_i| + |\gamma_j|, |\beta_i| + |\delta_j|\}\right\}}.
\end{equation}
Suppose that $P$ and $Q$ do not have a linear factor in common, so that $C > 0$. Then for all pairs $(a, b) \in \CC^2$ we have
$$
\max\{|P(a, b)|, |Q(a, b)|\} \geq C^rH(a, b)^r.
$$
\end{lem}

\begin{proof}
We claim that either
$$
\min\limits_{i = 1, \ldots, r}\{|\alpha_i a + \beta_i b|\} \geq C|b| \quad \text{or} \quad \min\limits_{j = 1, \ldots, r}\{|\gamma_ja + \delta_jb|\} \geq C|b|.
$$
For suppose not. Then for all $i, j$ we have
\begin{align*}
\left|(\alpha_i\delta_j - \beta_i\gamma_j)b\right|
& = |\alpha_i(\gamma_ja + \delta_jb) - \gamma_j(\alpha_ia + \beta_ib)|\\
& \leq (|\alpha_i| + |\gamma_j|)\max\left\{|\alpha_i a + \beta_i b|, |\gamma_j a + \delta_j b|\right\}\\
& < (|\alpha_i| + |\gamma_j|)C|b|\\
& \leq \min\{|\alpha_i\delta_j - \beta_i\gamma_j|\}|b|,
\end{align*}
so we reach a contradiction. Without loss of generality, suppose that $\min\{|\alpha_ia + \beta_i b|\}\geq C|b|$. Then
$$
|P(a, b)| = \prod\limits_{i = 1}^r|\alpha_ia + \beta_i b| \geq \min\left\{|\alpha_ia + \beta_ib|\right\}^r \geq C^r|b|^r.
$$
Analogously, either
$$
\min\limits_{i = 1, \ldots, r}\{|\alpha_i a + \beta_i b|\} \geq C|a| \quad \text{or} \quad \min\limits_{j = 1, \ldots, r}\{|\gamma_ja + \delta_jb|\} \geq C|a|.
$$
In the first case we can immediately conclude that $|P(a, b)| \geq C^rH(a, b)^r$ and the result follows. Otherwise we have $|Q(a, b)| \geq C^r|a|^r$. Combining this inequality with $|P(a, b)| \geq C^r|b|^r$ yields the result.
\end{proof}

For the proof of the following result, recall the definition of the \emph{Mahler measure} and the \emph{house} of a polynomial introduced in Section \ref{sec:theory}.

\begin{cor} \label{cor:two_forms_bound}
Let $P, Q \in \Z[x]$ be coprime polynomials of degrees $r$ and $s$, respectively, such that $r \geq \max\{1, s\}$. Define
$$
P(x, y) = y^rP(x/y) \quad \text{and} \quad Q(x, y) = y^rQ(x/y).
$$
Then for all pairs $(a, b) \in \CC^2$ we have
\small
$$
\max\{|P(a, b)|, |Q(a, b)|\} \geq \frac{H(a, b)^r}{2^{r^2} (r + 1)^{3r^2/2} \max\{H(P), H(Q)\}^{2r^2+r}}.
$$
\normalsize
\end{cor}

\begin{proof}
Let $c_P$ and $c_Q$ be the leading coefficients of $P$ and $Q$, respectively. Then
$$
|c_P| \cdot \max\{1, \house{P}\} \leq M(P) \quad \text{and} \quad |c_Q| \cdot \max\{1, \house{Q}\} \leq M(Q),
$$
and so it follows from \cite[Lemma 1.6.7]{bombieri-gubler} that
\small
\begin{equation} \label{eq:house-bounds}
|c_P| \cdot \max\{1, \house{P}\} \leq (r + 1)^{1/2}H(P) \quad \text{and} \quad |c_Q| \cdot \max\{1, \house{Q}\} \leq (s + 1)^{1/2}H(Q).
\end{equation}
\normalsize
Let $\mu_1, \ldots, \mu_r$ be the roots of $P(x)$ and write
$$
P(x, y) = c_P\prod\limits_{i = 1}^r(x - \mu_i y) = \prod\limits_{i = 1}^r(\alpha_ix + \beta_iy),
$$
where $\alpha_i = c_P^{1/r}$, $\beta_i = - c_P^{1/r}\mu_i$. We consider the following two cases.

\textbf{Case 1.} Suppose that $s = 0$, i.e., $Q(x) = c_Q$. Then
$$
Q(x, y) = c_Qy^r = \prod\limits_{j = 1}^r(\gamma_jx + \delta_jy),
$$
where $\gamma_j = 0$ and $\delta_j = c_Q^{1/r}$. Using (\ref{eq:house-bounds}), the constant $C$ in (\ref{eq:two_forms_bound_C}) can be estimated from below as follows:
\begin{align*}
C
& = \frac{|c_Pc_Q|^{1/r}}{\max_i\left\{\max\{|c_P|^{1/r}, |c_P|^{1/r}|\mu_i| + |c_Q|^{1/r}\}\right\}}\\
& = \frac{|c_Q|^{1/r}}{\max_i\{\ \max\{1, |\mu_i| + \left|\frac{c_Q}{c_P}\right|^{1/r}\}\ \}}\\
& = \frac{1}{\max_i\{\ \max\{\left|\frac{1}{c_Q}\right|^{1/r}, \left|\frac{\mu_i}{c_Q}\right|^{1/r} + \left|\frac{1}{c_P}\right|^{1/r}\}\ \}}\\
& = \frac{1}{\max\{\ \left|\frac{1}{c_Q}\right|^{1/r}, \frac{\house{P}}{|c_Q|^{1/r}} + \left|\frac{1}{c_P}\right|^{1/r} \ \}}\\
& \geq \frac{1}{|c_P|\cdot \max\{1, \house{P}\} + |c_Q| \cdot \max\{1, \house{Q}\}}\\
& \geq \frac{1}{2(r + 1)^{1/2}\max\{H(P), H(Q)\}}.
\end{align*}

\textbf{Case 2.} Suppose that $s \geq 1$. Let $\nu_1, \ldots, \nu_s$ be the roots of $Q(x)$ and write
$$
Q(x, y) = c_Qy^{r - s}\prod\limits_{j = 1}^s(x - \nu_jy) = \prod\limits_{j = 1}^r(\gamma_jx + \delta_jy),
$$
where
$$
\gamma_j =
\begin{cases}
c_Q^{1/r}, & \text{if $1 \leq i \leq s$,}\\
0, & \text{if $s + 1 \leq i \leq r$,}
\end{cases}
\quad
\text{and}
\quad
\delta_j =
\begin{cases}
-c_Q^{1/r}\nu_i, & \text{if $1 \leq i \leq s$,}\\
c_Q^{1/r}, & \text{if $s + 1 \leq i \leq r$.}
\end{cases}
$$
Using (\ref{eq:house-bounds}), the constant $C$ in (\ref{eq:two_forms_bound_C}) can be estimated from below as follows:
\begin{align*}
C
& = \frac{\min_{i, j}\{|\alpha_i\delta_j -\beta_i\gamma_j|\}}{\max_{i, j}\left\{\max\{|\alpha_i| + |\gamma_j|, |\beta_i| + |\delta_j|\}\right\}}\\
& \geq \frac{|c_Pc_Q|^{1/r}\min\{1, \min_{i, j}\{|\mu_i - \nu_j|\}\}}{|c_P|^{1/r}\cdot \max\{1, \house{P}\} + |c_Q|^{1/r}\cdot \max\{1, \house{Q}\}}\\
& \geq \frac{\min\{1, \min_{i, j}\{|\mu_i - \nu_j|\}\}}{2(r + 1)^{1/2}\max\{H(P), H(Q)\}}.
\end{align*}
By \cite[Theorem A]{bugeaud-mignotte}, 
$$
\min\limits_{\substack{1 \leq i \leq r\\ 1 \leq j \leq s}}\{|\mu_i - \nu_j|\} \geq 2^{1 - r}(r + 1)^{(1 - 3r)/2}\max\{H(P), H(Q)\}^{-2r}.
$$
Since $P, Q \in \Z[x]$ and $r \geq 1$, we have $\max\{H(P), H(Q)\} \geq 1$, so the quantity on the right-hand side of the above inequality does not exceed one. Combining the lower bound on $\min_{i, j}\{|\mu_i - \nu_j|\}$ with the lower bound on $C$ established above, we obtain
$$
C \geq 2^{-r} (r + 1)^{-3r/2} \max\{H(P), H(Q)\}^{-2r - 1}.
$$
The result now follows from Lemma \ref{lem:two_forms_bound}.
\end{proof}

\begin{proof}[Proof of Proposition \ref{prop:alternative_gap_principle}]
From equation (\ref{eq:R_vanishes_at_rational_point}) it follows that $Q(x_1/y_1) \neq 0$, for otherwise $P(x_1/y_1) = 0$, which means that $P$ and $Q$ are not coprime. Let
$$
P(x, y) = y^rP(x/y) \quad \text{and} \quad Q(x, y) = y^rQ(x/y).
$$
Since $|y_1| \geq 1$, it must be the case that $Q(x_1, y_1) = y_1^rQ(x_1/y_1) \neq 0$, so
$$
\frac{x_2}{y_2} = -\frac{P(x_1/y_1)}{Q(x_1/y_1)} = -\frac{P(x_1, y_1)}{Q(x_1, y_1)}.
$$
Since $x_2$ and $y_2$ are coprime, and $P(x_1, y_1)$ and $Q(x_1, y_1)$ are integers, we see that
$$
|x_2| = \frac{|P(x_1, y_1)|}{g} \quad \text{and} \quad |y_2| = \frac{|Q(x_1, y_1)|}{g},
$$
where $g = \gcd\left(P(x_1, y_1), Q(x_1, y_1)\right)$. By Lemma \ref{lem:resultant}, $g \leq (r + 1)^r\max\{H(P), H(Q)\}^{2r}$. Thus,
$$
H(x_2, y_2) = \frac{\max\{|P(x_1, y_1)|, |Q(x_1, y_1)|\}}{g} \geq \frac{\max\{|P(x_1, y_1)|, |Q(x_1, y_1)|\}}{(r + 1)^r\max\{H(P), H(Q)\}^{2r}}.
$$
Finally, since $P$ and $Q$ are coprime, Corollary \ref{cor:two_forms_bound} applies:
\begin{align*}
H(x_2, y_2)
& \geq \frac{\max\{|P(x_1, y_1)|, |Q(x_1, y_1)|\}}{(r + 1)^r\max\{H(P), H(Q)\}^{2r}}\\
& \geq \frac{H(x_1, y_1)^r}{2^{r^2} (r + 1)^{(3r^2 + 2r)/2} \max\{H(P), H(Q)\}^{2r^2 + 3r}}.
\end{align*}
\end{proof}

\section{A Generalized Archimedean Gap Principle} \label{sec:generalized_archimedean_gap_principle}

In this section we prove Theorem \ref{thm:archimedean_gap_principle}. Recall that, for any $h \in \Z[x]$ and $i$ such that $0 \leq i \leq \deg h$, the inequality
\begin{equation} \label{eq:Dih_bound}
|D_ih(\alpha)| \leq H(h)\binom{\deg h + 1}{i + 1}\max\{1, |\alpha|\}^{\deg h - i}
\end{equation}
holds.

Let $P, Q$ be a minimal pair for $(\alpha, \beta)$, and define $R(x, y) = P(x) + yQ(x)$, so that $R(\alpha, \beta) = 0$. Choose $C_1$ so that
$$
C_1 \geq C_0^{1/\mu}.
$$
Then
$$
\left|\alpha - \frac{x_1}{y_1}\right| < 1 \quad \text{and} \quad \left|\beta - \frac{x_2}{y_2}\right| < 1.
$$
If $R(x_1/y_1, x_2/y_2) \neq 0$, then it follows from the triangle inequality, (\ref{eq:Dih_bound}), and the two inequalities established above that
\begin{align*}
\frac{1}{H(x_1, y_1)^r H(x_2, y_2)}
& \leq \left|R\left(\frac{x_1}{y_1}, \frac{x_2}{y_2}\right)\right|\\
& \leq \sum\limits_{i = 0}^r\sum\limits_{j = 0}^1\left|D_{i,j}R(\alpha, \beta)\right|\left|\alpha - \frac{x_1}{y_1}\right|^i\left|\beta - \frac{x_2}{y_2}\right|^j\\
& < \frac{C_0}{H(x_1, y_1)^{\mu}}\sum\limits_{i = 0}^r\left(\left|D_{i, 0}R(\alpha, \beta)\right| + \left|D_{i, 1}R(\alpha, \beta)\right|\right)\\
& \leq \frac{C_0}{H(x_1, y_1)^{\mu}}\sum\limits_{i = 0}^r\left(|D_iP(\alpha)| + (1 + |\beta|)\cdot |D_iQ(\alpha)|\right)\\
& \leq \frac{C_0}{H(x_1, y_1)^{\mu}}\sum\limits_{i = 0}^r\left(\binom{r + 1}{i + 1}H(P) + (1 + |\beta|)\binom{r + 1}{i + 1}H(Q)\right)\max\{1, |\alpha|\}^{r - i}\\
& \leq \frac{C_0}{H(x_1, y_1)^{\mu}}(2 + |\beta|)\max\{H(P), H(Q)\}\max\{1, |\alpha|\}^r\sum\limits_{i = 0}^r\binom{r + 1}{i + 1}\\
& < \frac{C_0}{H(x_1, y_1)^{\mu}}2^{r + 1}(2 + |\beta|)C_{12}\max\{1, |\alpha|\}^r\\
& \leq \frac{C_0}{H(x_1, y_1)^{\mu}}2^{(d/2) + 1}(2 + |\beta|)C_{12}\max\{1, |\alpha|\}^{d/2}\\
& = \frac{C_2}{H(x_1, y_1)^{\mu}},
\end{align*}
where the second-to-last inequality follows from (\ref{eq:height_PQ}). By (\ref{eq:r_bound}),
$$
H(x_2,y_2) > C_2^{-1}H(x_1, y_1)^{\mu - r} \geq C_2^{-1}H(x_1, y_1)^{\mu - d/2},
$$
which means that case 1 holds.

Suppose that $R(x_1/y_1, x_2/y_2) = 0$. If $r = 1$, then by definition $R(x, y) = (sx + t) - y(ux + v)$ for some integers $s$, $t$, $u$ and $v$. Note that \mbox{$sv - tu \neq 0$}, for otherwise the number $\beta$ would have to be rational. Since \mbox{$R(\alpha, \beta) = 0$} and \mbox{$R(x_1/y_1,\ x_2/y_2) = 0$}, case 2 holds.

It remains to consider the case when $R(x_1/y_1, x_2/y_2) = 0$ and $r \geq 2$. We will prove that $H(x_1, y_1) < C$ for some positive real number $C$, which depends only on $\alpha$, $\beta$, $\mu$ and $C_0$. By choosing $C_1$ so that $C_1 \geq C$, we then arrive to a contradiction. Note that
\begin{equation} \label{eq:beta_minus_x2y2}
\left|\beta - \frac{x_2}{y_2}\right| = \left|\frac{P(\alpha)}{Q(\alpha)} - \frac{P(x_1/y_1)}{Q(x_1/y_1)}\right| = \frac{|P(\alpha)Q(x_1/y_1) - Q(\alpha)P(x_1/y_1)|}{|Q(\alpha)Q(x_1/y_1)|}.
\end{equation}
Further,
\begin{equation} \label{eq:Q_alpha_upper_bound}
|Q(\alpha)| \leq (r + 1)\max\{H(P), H(Q)\}\max\{1, |\alpha|\}^r \leq (r + 1)C_{12}\max\{1, |\alpha|\}^r,
\end{equation}
\begin{align} \label{eq:Q_x1y1_upper_bound}
\left|Q\left(\frac{x_1}{y_1}\right)\right|
& \leq \sum\limits_{i = 0}^r|D_iQ(\alpha)|\cdot\left|\alpha - \frac{x_1}{y_1}\right|^i\\\notag
& < \sum\limits_{i = 0}^r|D_iQ(\alpha)|\\\notag
& \leq H(Q)\sum\limits_{i = 0}^r\binom{r + 1}{i + 1}\max\{1, |\alpha|\}^{r - i}\\\notag
& < 2^{r + 1}C_{12}\max\{1, |\alpha|\}^r.\\\notag
\end{align}

It remains to estimate $|P(\alpha)Q(x_1/y_1) - Q(\alpha)P(x_1/y_1)|$ from below. Let $W = PQ' - QP'$ denote the Wronskian of $P$ and $Q$. By Taylor's Theorem, (\ref{eq:Dih_bound}) and (\ref{eq:W_alpha_archimedean_lower_bound}),
$$
\begin{array}{l}
\left|P(\alpha)Q\left(\frac{x_1}{y_1}\right) -Q(\alpha)P\left(\frac{x_1}{y_1}\right)\right|\\
= \left|P(\alpha)\sum\limits_{i = 0}^rD_iQ(\alpha)\left(\frac{x_1}{y_1} - \alpha\right)^i - Q(\alpha)\sum\limits_{i = 0}^rD_iP(\alpha)\left(\frac{x_1}{y_1} - \alpha\right)^i\right|\\
= \left|\alpha - \frac{x_1}{y_1}\right|\cdot \left|\sum\limits_{i = 0}^{r-1}\left(P(\alpha)D_{i+1}Q(\alpha) - Q(\alpha)D_{i+1}P(\alpha)\right)\left(\frac{x_1}{y_1} - \alpha\right)^i\right|\\
= \left|\alpha - \frac{x_1}{y_1}\right|\cdot \left|W(\alpha) + \left(\frac{x_1}{y_1} - \alpha\right)\sum\limits_{i = 1}^{r-1}\left(P(\alpha)D_{i+1}Q(\alpha) - Q(\alpha)D_{i+1}P(\alpha)\right)\left(\frac{x_1}{y_1} - \alpha\right)^{i-1}\right|\\
 > \left|\alpha - \frac{x_1}{y_1}\right|\left(|W(\alpha)| - \frac{C_0}{H(x_1, y_1)^\mu}\sum\limits_{i = 1}^{r-1}\left|P(\alpha)D_{i+1}Q(\alpha) - Q(\alpha)D_{i+1}P(\alpha)\right|\right)\\
\geq \left|\alpha - \frac{x_1}{y_1}\right|\left(|W(\alpha)| - \frac{C_0}{H(x_1, y_1)^\mu}2(r + 1)\max\{H(P), H(Q)\}^2\sum\limits_{i = 1}^{r - 1}\binom{r + 1}{i + 2}\max\{1, |\alpha|\}^{2r - i - 1}\right)\\
\geq \left|\alpha - \frac{x_1}{y_1}\right|\left(C_{13} - C_1^{-\mu}C_02^{r + 2}(r + 1)C_{12}^2\max\{1, |\alpha|\}^{2r}\right),
\end{array}
$$
where the last inequality follows from $H(x_1, y_1) \geq C_1$, (\ref{eq:height_PQ}) and (\ref{eq:W_alpha_archimedean_lower_bound}). Thus, if we choose $C_1$ so that
$$
C_1^\mu \geq 2^{(d/2) + 3}((d/2) + 1)C_0C_{12}^2C_{13}^{-1}\max\{1, |\alpha|\}^d,
$$
then it follows from $r \leq d/2$ that
\begin{align*}
\left|P(\alpha)Q\left(\frac{x_1}{y_1}\right) -Q(\alpha)P\left(\frac{x_1}{y_1}\right)\right|
& \geq \left|\alpha - \frac{x_1}{y_1}\right|\left(C_{13} - C_1^{-\mu}C_02^{r + 2}(r + 1)C_{12}^2\max\{1, |\alpha|\}^{2r}\right)\\
& \geq \left|\alpha - \frac{x_1}{y_1}\right|\left(C_{13} - C_1^{-\mu}C_02^{(d/2) + 2}((d/2) + 1)C_{12}^2\max\{1, |\alpha|\}^d\right)\\
& \geq \frac{C_{13}}{2}\left|\alpha - \frac{x_1}{y_1}\right|.
\end{align*}
Combining the above result with (\ref{eq:beta_minus_x2y2}), (\ref{eq:Q_alpha_upper_bound}) and (\ref{eq:Q_x1y1_upper_bound}) yields
$$
\left|\beta - \frac{x_2}{y_2}\right|
= \frac{|P(\alpha)Q(x_1/y_1) - Q(\alpha)P(x_1/y_1)|}{|Q(\alpha)Q(x_1/y_1)|}\\
> \frac{C_{13}}{2^{r + 2}(r + 1)C_{12}^2\max\{1, |\alpha|\}^{2r}}\left|\alpha - \frac{x_1}{y_1}\right|.
$$
By Proposition \ref{prop:alternative_gap_principle},
$$
H(x_2, y_2) \geq \frac{H(x_1, y_1)^r}{C_{15}\max\{H(P), H(Q)\}^{2r^2 + 3r}} \geq \frac{H(x_1, y_1)^r}{C_{15}C_{12}^{2r^2 + 3r}},
$$
where
\begin{equation} \label{eq:C15-upper-bound}
C_{15} = 2^{r^2}(r + 1)^{(3r^2 + 2r)/2} \leq 2^{d^2/4} ((d/2) + 1)^{(3d^2+4d)/8}.
\end{equation}
Consequently,
\begin{align*}
\left|\alpha - \frac{x_1}{y_1}\right|
& < \frac{2^{r + 2}(r + 1)C_{12}^2\max\{1, |\alpha|\}^{2r}}{C_{13}}\left|\beta - \frac{x_2}{y_2}\right|\\
& < \frac{2^{r + 2}(r + 1)C_{12}^2\max\{1, |\alpha|\}^{2r}}{C_{13}} \cdot \frac{C_0}{H(x_2, y_2)^\mu}\\
& \leq \frac{2^{r + 2}(r + 1)C_0 C_{12}^{(2r^2 + 3r)\mu + 2} C_{15}^\mu\max\{1, |\alpha|\}^{2r}}{C_{13}H(x_1, y_1)^{r\mu}}.
\end{align*}
Thus, we obtain an upper bound on $|\alpha - x_1/y_1|$. On the other hand, by \mbox{Lemma \ref{lem:liouville}}, the lower bound (\ref{eq:liouville_inequality}) holds. Combining upper and lower bounds,
$$
\frac{C_6}{H(x_1, y_1)^d} \leq \left|\alpha - \frac{x_1}{y_1}\right| < \frac{2^{r + 2}(r + 1)C_0 C_{12}^{(2r^2 + 3r)\mu + 2} C_{15}^\mu\max\{1, |\alpha|\}^{2r}}{C_{13}H(x_1, y_1)^{r\mu}}
$$
Since $\mu > (d/2) + 1$ and $r \geq 2$, we see that $r\mu - d > 0$, so
\begin{align*}
H(x_1, y_1)
& < \left(2^{r + 2}(r + 1)C_0(C_6C_{13})^{-1}C_{12}^{(2r^2 + 3r)\mu + 2} C_{15}^\mu \max\{1, |\alpha|\}^{2r}\right)^{1/(r\mu - d)}\\
& \leq \left(2^{d^2\mu/4} ((d/2) + 1)^{(3d^2+4d)\mu/8}C_0(C_6C_{13})^{-1}C_{12}^{(d^2 + 3d)\mu/2 + 2}\max\{1, |\alpha|\}^d\right)^{1/(2\mu - d)}\\
& = C.
\end{align*}
Notice how in the second-to-last inequality we have utilized the upper bound on $C_{15}$ given in (\ref{eq:C15-upper-bound}). Thus, if we choose $C_1$ so that $C_1 \geq C$, then $H(x_1, y_1) \geq C_1 \geq C$, and so we arrive to a contradiction.

\section{A Generalized Non-Archimedean Gap Principle} \label{sec:generalized_non-archimedean_gap_principle}

In this section we prove Theorem \ref{thm:non-archimedean_gap_principle}. Let $P, Q$ be a minimal pair for $(\alpha, \beta)$, and define $R(x, y) = P(x) + yQ(x)$, so that $R(\alpha, \beta) = 0$. Suppose that \mbox{$R(x_1/y_1, x_2/y_2) \neq 0$.} Then the following trivial lower bound holds:
\begin{align*}
\left|y_1^ry_2R\left(\frac{x_1}{y_1}, \frac{x_2}{y_2}\right)\right|_p
& \geq \frac{1}{|y_1^ry_2R(x_1/y_1, x_2/y_2)|}\\
& \geq \frac{1}{2(r + 1)\max\{H(P), H(Q)\}H(x_1, y_1)^rH(x_2, y_2)}\\
& \geq \frac{1}{2((d/2) + 1)C_{12}H(x_1, y_1)^{d/2}H(x_2, y_2)}.
\end{align*}
Let $c_\alpha$ and $c_\beta$ denote the leading coefficients of the minimal polynomials of $\alpha$ and $\beta$, respectively. Note that for each $(i, j) \in \{0, \ldots, r\} \times \{0, 1\}$ the $p$-adic number $c_\alpha^{r - i}c_\beta^{1 - j}D_{ij}R(\alpha, \beta)$ is an algebraic integer. Thus, its $p$-adic absolute value does not exceed one. Via the application of Taylor's Theorem we obtain the following upper bound:
\begin{align*}
\left|y_1^ry_2R\left(\frac{x_1}{y_1}, \frac{x_2}{y_2}\right)\right|_p
& \leq \max\limits_{(i, j) \neq (0, 0)}\left\{\left|D_{ij}R(\alpha, \beta)\right|_p \cdot |y_1\alpha - x_1|_p^i\cdot |y_2\beta -x_2|_p^j\right\}\\
& = \max\limits_{(i, j) \neq (0, 0)}\left\{\left|\frac{c_\alpha^{r - i}c_\beta^{1 - j}D_{ij}R(\alpha, \beta)}{c_\alpha^{r - i}c_\beta^{1 - j}}\right|_p \cdot |y_1\alpha - x_1|_p^i\cdot |y_2\beta -x_2|_p^j\right\}\\
& \leq |c_\alpha^rc_\beta|_p^{-1} \max\limits_{(i, j) \neq (0, 0)}\left\{|y_1c_\alpha\alpha - c_\alpha x_1|_p^i\cdot |y_2c_\beta\beta -c_\beta x_2|_p^j\right\}\\
& \leq c_\alpha^rc_\beta\max\{|y_1c_\alpha\alpha - c_\alpha x_1|_p, |y_2c_\beta\beta - c_\beta x_2|_p\}\\
& \leq c_\alpha^rc_\beta\max\{|c_\alpha|_p, |c_\beta|_p\}\max\{|y_1\alpha - x_1|_p, |y_2\beta - x_2|_p\}\\
& < \frac{C_0c_\alpha^rc_\beta}{H(x_1, y_1)^\mu}.
\end{align*}
Upon combining the upper and lower bounds, we obtain
$$
\frac{1}{2((d/2) + 1)C_{12}H(x_1, y_1)^{d/2}H(x_2, y_2)} < \frac{C_0c_\alpha^r c_\beta}{H(x_1, y_1)^\mu}.
$$
If we now set $C_4 = (d + 2)C_0C_{12}c_\alpha^{d/2}c_\beta$, then
$$
H(x_2, y_2) > C_4^{-1}H(x_1, y_1)^{\mu - d/2},
$$
which means that case 1 holds.

Suppose that $R(x_1/y_1, x_2/y_2) = 0$. If $r = 1$, then by definition $R(x, y) = (sx + t) - y(ux + v)$ for some integers $s, t, u, v$. Note that $sv - tu \neq 0$, for otherwise the number $\beta$ would have to be rational. Since \mbox{$R(\alpha, \beta) = 0$} and \mbox{$R(x_1/y_1, x_2/y_2) = 0$}, case 2 holds.

It remains to consider the case when $R(x_1/y_1, x_2/y_2) = 0$ and $r \geq 2$. We will prove that $H(x_1, y_1) < C$ for some positive number $C$, which depends only on $\alpha$ and $\beta$. By choosing $C_3$ so that $C_3 \geq C$, we then arrive to a contradiction.

First, we claim that by choosing $C_3$ so that
$$
C_3 \geq C_0^{1/\mu}
$$
we can ensure that $|y_1|_p \geq c_\alpha^{-1}$. This inequality clearly holds when $p$ does not divide $y_1$, so we assume that $p \mid y_1$. Since $x_1$ and $y_1$ are coprime, it must be the case that $p$ does not divide $x_1$. Suppose that $|y_1\alpha|_p \neq |x_1|_p$. Then it follows from the strong triangle inequality that
$$
|y_1\alpha - x_1|_p = \max\left\{|y_1\alpha|_p,\ |x_1|_p\right\} = \max\left\{|y_1\alpha|_p,\ 1\right\} \geq 1.
$$
Since $|y_1\alpha - x_1|_p < C_0H(x_1, y_1)^{-\mu}$, we find that $H(x_1, y_1) < C_0^{1/\mu} \leq C_3$, so we reach a contradiction. Thus, $|y_1\alpha|_p = |x_1|_p = 1$. Since $c_\alpha \alpha$ is an algebraic integer, it must be the case that $|c_\alpha \alpha|_p \leq 1$, so
$$
|y_1|_p = \left|\alpha^{-1}\right|_p \geq \left|c_\alpha\right|_p \geq c_\alpha^{-1},
$$
as claimed.

Now that we have chosen $C_3$ so that $|y_1|_p \geq c_\alpha^{-1}$, we turn our attention to the equation $R(x_1/y_1, x_2/y_2) = 0$, which implies that
$$
|x_2| = \frac{|P(x_1, y_1)|}{g} \quad \text{and} \quad |y_2| = \frac{|Q(x_1, y_1)|}{g},
$$
where $g = \gcd(P(x_1, y_1), Q(x_1, y_1))$. Consequently,
\begin{equation} \label{eq:non-archimedean_beta_minus_x2y2}
|y_2\beta - x_2|_p = \frac{|P(\alpha)Q(x_1, y_1) - Q(\alpha)P(x_1, y_1)|_p}{|gQ(\alpha)|_p}.
\end{equation}
Since $g$ is an integer, we have
\begin{equation} \label{eq:g_is_an_integer}
|g|_p \leq 1.
\end{equation}
Since $c_\alpha^r Q(\alpha)$ is an algebraic integer,
\begin{equation} \label{eq:non-archimedean_Q_alpha_bound}
|Q(\alpha)|_p \leq |c_\alpha|_p^{-r} \leq c_\alpha^r \leq c_\alpha^{d/2}.
\end{equation}

It remains to estimate $|P(\alpha)Q(x_1, y_1) - Q(\alpha)P(x_1, y_1)|_p$ from below. Before we proceed, note that for any $i$ the number
$$
c_\alpha^{2r-i-1}\left(P(\alpha)D_{i+1}Q(\alpha) - Q(\alpha)D_{i+1}P(\alpha)\right)
$$
is an algebraic integer, so its $p$-adic absolute value does not exceed one. Consequently,
\begin{align*}
\left|\frac{P(\alpha)D_{i+1}Q(\alpha) - Q(\alpha)D_{i+1}P(\alpha)}{c_\alpha^i}\right|_p
& = \left|\frac{c_\alpha^{2r-1-i}\left(P(\alpha)D_{i+1}Q(\alpha) - Q(\alpha)D_{i+1}P(\alpha)\right)}{c_\alpha^{2r-1}}\right|_p\\
& \leq |c_{\alpha}|_p^{-(2r-1)}\\
& \leq c_\alpha^{2r - 1}\\
& \leq c_\alpha^{d - 1}.
\end{align*}
Now, let $W = PQ' - QP'$ denote the Wronskian of $P$ and $Q$. By Taylor's Theorem and (\ref{eq:W_alpha_non-archimedean_lower_bound}),
\small
$$
\begin{array}{l}
\left|P(\alpha)Q(x_1, y_1) - Q(\alpha)P(x_1, y_1)\right|_p\\
= \left|P(\alpha)\sum\limits_{i = 0}^rD_iQ(\alpha)\left(x_1 - \alpha y_1\right)^iy_1^{r-i} - Q(\alpha)\sum\limits_{i = 0}^rD_iP(\alpha)\left(x_1 - \alpha y_1\right)^iy_1^{r-i}\right|_p\\
= \left|y_1\alpha - x_1\right|_p\left|\sum\limits_{i = 0}^{r-1}\left(P(\alpha)D_{i+1}Q(\alpha) - Q(\alpha)D_{i+1}P(\alpha)\right)\left(x_1 - \alpha y_1\right)^iy_1^{r-1-i}\right|_p\\
\geq \left|y_1\alpha - x_1\right|_p\left(|W(\alpha)y_1^{r - 1}|_p - |y_1\alpha - x_1|_p\max\limits_{i = 0, \ldots, r - 1}\left\{\left|\frac{P(\alpha)D_{i+1}Q(\alpha) - Q(\alpha)D_{i+1}P(\alpha)}{c_\alpha^i}\right|_p\right\}\right)\\
> \left|y_1\alpha - x_1\right|_p\left(|W(\alpha)y_1^{r - 1}|_p - \frac{C_0}{H(x_1, y_1)^\mu}c_\alpha^{2r - 1}\right)\\
\geq \left|y_1\alpha - x_1\right|_p\left(C_{14}c_\alpha^{-(r - 1)} - C_3^{-\mu}C_0c_\alpha^{2r - 1}\right),
\end{array}
\normalsize
$$
where the last inequality follows from $H(x_1, y_1) \geq C_3$, (\ref{eq:W_alpha_non-archimedean_lower_bound}) and $|y_1|_p \geq c_\alpha^{-1}$. Thus, if we choose $C_3$ so that
$$
C_3^\mu \geq 2C_0C_{14}^{-1}c_\alpha^{(3d-4)/2},
$$
then
$$
\left|P(\alpha)Q(x_1, y_1) - Q(\alpha)P(x_1, y_1)\right|_p > \left|y_1\alpha - x_1\right|_p\left(C_{14}c_\alpha^{-(r - 1)} - C_3^{-\mu}C_0c_\alpha^{2r - 1}\right) \geq \frac{C_{14}}{2c_\alpha^{(d-2)/2}}\left|y_1\alpha - x_1\right|_p.
$$
Combining this observation with (\ref{eq:non-archimedean_beta_minus_x2y2}), (\ref{eq:g_is_an_integer}) and (\ref{eq:non-archimedean_Q_alpha_bound}),
$$
\left|y_2\beta - x_2\right|_p = \frac{|P(\alpha)Q(x_1, y_1) - Q(\alpha)P(x_1, y_1)|_p}{|gQ(\alpha)|_p} > \frac{C_{14}}{2c_\alpha^{d - 1}}|y_1\alpha - x_1|_p.
$$
By Proposition \ref{prop:alternative_gap_principle}, 
$$
H(x_2, y_2) \geq \frac{H(x_1, y_1)^r}{C_{15}\max\{H(P), H(Q)\}^{2r^2 + 3r}} \geq \frac{H(x_1, y_1)^r}{C_{15}C_{12}^{2r^2 + 3r}},
$$
where $C_{15}$ satisfies the upper bound (\ref{eq:C15-upper-bound}). Consequently,
$$
\left|y_1\alpha - x_1\right|_p
< \frac{2c_\alpha^{d - 1}}{C_{14}}\left|y_2\beta - x_2\right|_p
< \frac{2c_\alpha^{d - 1}C_0}{C_{14}H(x_2, y_2)^\mu}
\leq \frac{2c_\alpha^{d - 1}C_0C_{12}^{(2r^2 + 3r)\mu}C_{15}^\mu}{C_{14}H(x_1, y_1)^{r\mu}}
$$
Thus, we obtain an upper bound on $|y_1\alpha - x_1|_p$. On the other hand, by Lemma \ref{lem:p-adic_liouville} we have the lower bound (\ref{eq:p-adic_liouville_inequality}). Combining upper and lower bounds,
$$
\frac{C_7}{H(x_1, y_1)^d} \leq \left|y_1\alpha - x_1\right|_p < \frac{2c_\alpha^{d-1}C_0C_{12}^{(2r^2 + 3r)\mu}C_{15}^\mu}{C_{14}H(x_1, y_1)^{r\mu}}
$$
Since $\mu > (d/2) + 1$ and $r \geq 2$, we have
\begin{align*}
H(x_1, y_1)
& < \left(2c_\alpha^{d-1}C_0C_7^{-1}C_{12}^{(2r^2 + 3r)\mu} C_{14}^{-1} C_{15}^\mu\right)^{1/(r\mu - d)}\\
& \leq \left(2^{d^2\mu/4} ((d/2) + 1)^{(3d^2+4d)\mu/8}c_\alpha^{d-1}C_0C_7^{-1}C_{12}^{(d^2+3d)\mu/2}C_{14}^{-1}\right)^{1/(2\mu - d)}\\
& = C.
\end{align*}
Notice how in the second-to-last inequality we have utilized the upper bound on $C_{15}$ given in (\ref{eq:C15-upper-bound}). Thus, if we choose $C_3$ so that $C_3 \geq C$, then $H(x_1, y_1) \geq C_3 \geq C$, and so we arrive to a contradiction.

\section{The Enhanced Automorphism Group} \label{sec:automorphisms}

In this section we establish several results about the enhanced automorphism group $\Aut' |F|$ of a binary form $F$. At the end we prove Proposition \ref{prop:automorphisms}, where we explain the relation between automorphisms of $F$ and the roots of $F(x, 1)$. 

\begin{lem} \label{lem:Aut_QF_is_finite}
Let $F \in \Z[x, y]$ be a binary form of degree $d \geq 3$ and nonzero discriminant $D(F)$. Then $\Aut_{\Q} |F|$ is $\GL_2(\Q)$-conjugate to one of the groups from \mbox{Table \ref{tab:finite_subgroups}}.
\end{lem}

\begin{proof}
See \cite{newman}.
\end{proof}

\begin{table}[t]
\centering
\begin{tabular}{| l | l | l | l |}
\hline
Group & Generators & Group & Generators\\
\hline\rule{0pt}{4ex}
$\bm C_1$ &
$\begin{pmatrix}1 & 0\\0 & 1\end{pmatrix}$ &
$\bm D_1$ &
$\begin{pmatrix}0 & 1\\1 & 0\end{pmatrix}$\\\rule{0pt}{5ex}
$\bm C_2$ &
$\begin{pmatrix}-1 & 0\\0 & -1\end{pmatrix}$ &
$\bm D_2$ &
$\begin{pmatrix}0 & 1\\1 & 0\end{pmatrix}, \begin{pmatrix}-1 & 0\\0 & -1\end{pmatrix}$\\\rule{0pt}{5ex}
$\bm C_3$ &
$\begin{pmatrix}0 & 1\\-1 & -1\end{pmatrix}$ &
$\bm D_3$ &
$\begin{pmatrix}0 & 1\\1 & 0\end{pmatrix}, \begin{pmatrix}0 & 1\\-1 & -1\end{pmatrix}$\\\rule{0pt}{5ex}
$\bm C_4$ &
$\begin{pmatrix}0 & 1\\-1 & 0\end{pmatrix}$ &
$\bm D_4$ &
$\begin{pmatrix}0 & 1\\1 & 0\end{pmatrix}, \begin{pmatrix}0 & 1\\-1 & 0\end{pmatrix}$\\\rule{0pt}{5ex}
$\bm C_6$ &
$\begin{pmatrix}0 & -1\\1 & 1\end{pmatrix}$ &
$\bm D_6$ &
$\begin{pmatrix}0 & 1\\1 & 0\end{pmatrix}, \begin{pmatrix}0 & 1\\-1 & 1\end{pmatrix}$\\
\hline
\end{tabular}
\caption{Representatives of equivalence classes of finite subgroups of $\GL_2(\Q)$ under $\GL_2(\Q)$-conjugation.}
\label{tab:finite_subgroups}
\end{table}

\begin{lem} \label{lem:G_is_finite}
Let $F \in \Z[x, y]$ be a binary form of degree $d \geq 3$ and nonzero discriminant $D(F)$. Let $\Aut' |F|$ be as in (\ref{eq:G}). Then $\Aut' |F| \cong \bm C_n$ or $\Aut'|F| \cong \bm D_n$, where $n \in \{1, 2, 3, 4, 6, 8, 12\}$.
\end{lem}

\begin{proof}
Note that $\Aut_{\Q} |F|$ is a subgroup of $\Aut' |F|$. Furthermore, for any $M \in \Aut' |F|$ we have $M^2 \in \Aut_{\mathbb Q}|F|$. By Lemma \ref{lem:Aut_QF_is_finite}, $\Aut_{\mathbb Q}|F|$ is finite, and so any $M \in \Aut' |F|$ has finite order. In fact, since the orders of elements in $\Aut_{\mathbb Q}|F|$ are $\{1, 2, 3, 4, 6\}$, the only possible orders of elements in $\Aut'|F|$ are $\{1, 2, 3, 4, 6, 8, 12\}$.

Next, recall a classical result that any finite subgroup of $\GL_2(\R)$ is $\GL_2(\R)$-conjugate to a finite subgroup of the orthogonal group $O_2(\R)$. 
Since finite subgroups of $O_2(\R)$ correspond to rotations and reflections on a plane, we conclude that each finite subgroup of $\GL_2(\R)$, including $\Aut' |F|$, is isomorphic to either a cyclic group $\bm C_n$ of order $n$ or a dihedral group $\bm D_n$ of order $2n$.

Now suppose that $\Aut' |F|$ contains at least $25$ distinct elements $M_1, \ldots, M_{25}$. By Schur's Theorem \cite{curtis-reiner}, any finitely generated torsion subgroup of $\GL_n(\CC)$ is finite. Hence $\langle M_1, \ldots, M_{25}\rangle$ is a finite subgroup of $\GL_2(\R)$, so it is isomorphic to either $\bm C_n$ or $\bm D_n$ for some $n$. In the former case we see that $n \geq 25$, while in the latter case $n \geq 13$. In both cases we obtain a contradiction, since the largest order that an element of $\Aut' |F|$ can have is $12$. Therefore $\Aut' |F|$ contains at most $24$ elements.
\end{proof}

Let us give an example of a group of the form (\ref{eq:G}) that is not a subgroup of $\GL_2(\Q)$. Consider
$$
G = \left\langle
\begin{pmatrix}
0 & 1\\
1 & 0
\end{pmatrix},
\begin{pmatrix}
1/\sqrt 3 & 1/\sqrt 3\\
-1/\sqrt 3 & 2/\sqrt 3
\end{pmatrix}
\right\rangle.
$$
Then $G \cong \bm D_{12}$. If we choose coprime integers $a$ and $b$ so that $a \equiv 3b \pmod{10}$, then any (reciprocal) binary form
\begin{align*}
F(x, y)
& = a(x^{12} + y^{12}) -6axy(x^{10} + y^{10})\\
& + \frac{231a + 2b}{5}x^2y^2(x^8 + y^8) - (176a + 2b)x^3y^3(x^6 + y^6)\\
& + \frac{495a + 5b}{2}x^4y^4(x^4 + y^4) + 2bx^5y^5(x^2 + y^2)\\
& - \frac{1122a+29b}{5}x^6y^6
\end{align*}
will have integer coefficients and satisfy $F_M = F$ for any $M \in G$. Consequently, if $(x, y)$ is a solution to the Thue equation $F(x, y) = m$, then so are $(y, -x + y)$, $(-x + y, -x)$, $(-x, -y)$, \mbox{$(-y, x - y)$,} $(x - y, x), (y, x), (-x + y, y), (-x, -x + y), (-y, -x), (x-y, -y), (x, x-y)$. This phenomenon was observed by Stewart in \mbox{\cite[Section 6]{stewart2}} with respect to binary forms invariant under $\bm D_6$, which is a subgroup of $G$. In addition to these $12$ solutions, we have $F(x', y') = 729m$ for any $(x', y') \in \left\{(x + y, -x + 2y)\right.$, $(-x + 2y, -2x + y)$, $(-2x + y, -x - y)$, $(-x - y, x - 2y)$, $(x - 2y, 2x - y)$, $(2x - y, x + y)$, $(-x + 2y, x + y)$, $(-2x + y, -x + 2y)$, $(-x - y, -2x + y)$, $(x - 2y, -x - y)$, $(2x - y, x - 2y)$, $\left.(x + y, 2x - y)\right\}$.

\begin{prop} \label{prop:automorphisms}
Let $F(x, y) = c_dx^d + c_{d-1}x^{d-1}y + \cdots + c_0y^d \in \Z[x, y]$ be an irreducible binary form of degree $d \geq 3$. Let $\alpha_1, \ldots, \alpha_d$ be the roots of $F(x, 1)$. There exists an index $j \in \{1, \ldots, d\}$ such that
$$
\alpha_j = \frac{v\alpha_1 - u}{-t\alpha_1 + s}
$$
for some integers $s$, $t$, $u$ and $v$ if and only if the matrix
$$
M = \frac{1}{\sqrt{|sv - tu|}}
\begin{pmatrix}
s & u\\
t & v
\end{pmatrix}
$$
is an element of $\Aut'|F|$. Furthermore, if $M \in \Aut'|F|$, then $|sv - tu| = \left|\frac{F(s, t)}{c_d}\right|^{2/d}$.
\end{prop}

\begin{proof}
Suppose that there exists an index $j \in \{1, \ldots, d\}$ such that $\alpha_j = \frac{v\alpha_1 - u}{-t\alpha_1 + s}$ for some integers $s$, $t$, $u$ and $v$. Since $F(x, 1)$ is irreducible, its Galois group acts transitively on the roots $\alpha_1, \alpha_2, \ldots, \alpha_d$. Therefore,
$$
\frac{v\alpha_1 - u}{-t\alpha_1 + s}, \quad \frac{v\alpha_2 - u}{-t\alpha_2 + s}, \quad \ldots, \quad \frac{v\alpha_d - u}{-t\alpha_d + s}
$$
is a permutation of $\alpha_1, \ldots, \alpha_d$. Thus,
\begin{align*}
F(x, y)
& = c_d\prod\limits_{i = 1}^d\left(x - \frac{v\alpha_i - u}{-t\alpha_i + s}y\right)\\
& = \frac{c_d}{\prod_{i = 1}^d(s - t\alpha_i)}\prod\limits_{i = 1}^d\left((-t\alpha_i + s)x - (v\alpha_i - u)y\right)\\
& = \frac{c_d}{F(s, t)}F(sx + uy,\ tx + vy)\\
& = \pm \eta^dF(sx + uy,\ tx + vy)\\
& = \pm F_M(x, y),
\end{align*}
where $\eta = |c_d/F(s, t)|^{1/d}$ and $M = \eta\left(\begin{smallmatrix}s & u\\t & v\end{smallmatrix}\right)$. Since
$$
D(F_M) = (\det M)^{d(d - 1)}D(F)
$$
and $F_M = \pm F$, we see that $|D(F_M)| = |D(F)|$, so $|\det M| = 1$. Hence $|\eta|^2\cdot |sv - tu| = 1$, which leads us to the conclusion that $\eta = |\eta| = |sv - tu|^{-1/2}$ and $M \in \operatorname{Aut}'|F|$.

Conversely, suppose that $M = |sv - tu|^{-1/2}\left(\begin{smallmatrix}s & u\\t & v\end{smallmatrix}\right)$ is in $\Aut' |F|$. Then
\begin{align*}
\pm F(x, y) & = F_M(x, y)\\
& = \frac{c_d}{|sv - tu|^{d/2}}\prod\limits_{i = 1}^d(sx + uy - \alpha_i(tx + vy))\\
& = \frac{F(s, t)}{|sv - tu|^{d/2}}\prod\limits_{i = 1}^d\left(x - \frac{v\alpha_i - u}{-t\alpha_i + s}y\right).
\end{align*}
We see that the polynomial $F_M(x, 1)$ vanishes at $\frac{v\alpha_i - u}{-t\alpha_i + s}$ for $i = 1, \ldots, d$. Since $F_M = \pm F$, the polynomials $F_M(x, 1)$ and $F(x, 1)$ have the same roots, so there exists some index $j$ such that $\alpha_j = \frac{v\alpha_1 - u}{-t\alpha_1 + s}$. Furthermore, the leading coefficients of $F(x, 1)$ and $F_M(x, 1)$ must equal up to a sign, i.e., $|c_d| = \frac{|F(s, t)|}{|sv - tu|^{d/2}}$. Of course, this is the same as $|sv - tu| = \left|\frac{F(s, t)}{c_d}\right|^{2/d}$.
\end{proof}

\section{Counting Primitive Solutions of Large Height to Certain Thue Inequalities} \label{sec:thue-inequality}
In this section we prove Theorem \ref{thm:thue_large}. It follows from a more general result stated in Theorem \ref{thm:count_large_approximations}, where we count approximations $x/y$ of large height to distinct algebraic numbers $\alpha_1, \ldots, \alpha_n$ such that $\Q(\alpha_i) = \Q(\alpha_1)$ for all $i = 1, 2, \ldots, n$. In order to state the main result of this section we need to introduce the notion of an \emph{orbit}. For an irrational number $\alpha$, the \emph{orbit} of $\alpha$ is the set
$$
\orb(\alpha) = \left\{\frac{v\alpha - u}{-t\alpha + s}\ \colon\ s,t,u,v \in \Z,\ sv-tu \neq 0\right\}.
$$

\begin{thm} \label{thm:count_large_approximations}
Let $K = \CC$ or $\Q_p$, where $p$ is a rational prime, and denote the standard absolute value on $K$ by $|\quad|$. Let $\alpha_1 \in K$ be an algebraic number of degree $d \geq 3$ over $\Q$ and $\alpha_2, \alpha_3, \ldots, \alpha_n$ be distinct elements of $\Q(\alpha_1)$, different from $\alpha_1$, each of \mbox{degree $d$.} Let $\mu$ be such that $(d/2) + 1 < \mu < d$. Let $C_0$ be a real number such that $C_0 > (4e^A)^{-1}$, where
\begin{equation} \label{eq:A}
A = 500^2\left(\log \max\limits_{i = 1, \ldots, n}\{M(\alpha_i)\} + \frac{d}{2}\right).
\end{equation}
There exists a positive real number $C_{16}$, which depends on $\alpha_1, \alpha_2, \ldots, \alpha_n$, $\mu$ and $C_0$, with the following property. The total number of rationals $x/y$ in lowest terms, which satisfy \mbox{$H(x, y) \geq C_{16}$} and
\begin{equation} \label{eq:roths_inequality}
\left|\alpha_j - \frac{x}{y}\right| < \frac{C_0}{H(x, y)^\mu}
\end{equation}
for some $j \in \{1, 2, \ldots, n\}$ is less than
$$
\gamma\left\lfloor1 + \frac{11.51 + 1.5\log d + \log \mu}{\log(\mu - d/2)}\right\rfloor,
$$
where
\begin{equation} \label{eq:gamma}
\gamma = \max\{\gamma_1, \ldots, \gamma_n\}, \quad \gamma_i = \#\{j \colon \alpha_j \in \orb(\alpha_i)\}.
\end{equation}
\end{thm}

Let us see why Theorem \ref{thm:thue_large} follows from Theorem \ref{thm:count_large_approximations}.

\begin{proof}[Proof of Theorem \ref{thm:thue_large}]
Let $\alpha_1, \alpha_2, \ldots, \alpha_d$ be the roots of $F(x, 1)$. Notice that, since $F(x, y)$ is irreducible, the roots of $F(1, x)$ are given by $\alpha_1^{-1}, \ldots, \alpha_d^{-1}$. Furthermore, since the field extension $\mathbb Q(\alpha)/\mathbb Q$ is Galois, we have $\mathbb Q(\alpha_i) = \mathbb Q(\alpha_1)$ for all $i = 1, \ldots, d$.

Choose $C_5$ so that
$$
C_5^{d - \mu} > \frac{2^{d-1}d^{(d-1)/2}M(F)^{d-2}m}{|D(F)|^{1/2}}.
$$
Let $(x, y)$ be a primitive solution to (\ref{eq:thue-inequality}) such that $H(x, y) \geq C_5$. Then it follows from the result of Lewis and Mahler stated in Lemma \ref{lem:lewis-mahler} that there exists an index $j \in \{1, 2, \ldots, d\}$ such that
$$
\min\left\{\left|\alpha_j - \frac{x}{y}\right|, \left|\alpha_j^{-1} - \frac{y}{x}\right|\right\} \leq \frac{2^{d-1}d^{(d-1)/2}M(F)^{d-2}m}{|D(F)|^{1/2}H(x, y)^d} < \frac{1}{H(x, y)^\mu}.
$$
Next, adjust the choice of $C_5$ so that Theorem \ref{thm:count_large_approximations} applies:
$$
C_5 \geq \max\left\{C_{16}(\alpha_1, \ldots, \alpha_d, \mu, C_0),\ C_{16}(\alpha_1^{-1}, \ldots, \alpha_d^{-1}, \mu, C_0)\right\},
$$
where $C_0 = 1$. If we let $\gamma$ be as in (\ref{eq:gamma}), then it follows from Theorem \ref{thm:count_large_approximations} that $x/y$ is one of at most
$$
2\gamma\left\lfloor1 + \frac{11.51 + 1.5\log d + \log \mu}{\log(\mu - d/2)}\right\rfloor
$$
rationals in lowest terms that satisfy either of the two inequalities
$$
\left|\alpha_j - \frac{x}{y}\right| < \frac{C_0}{H(x, y)^\mu}, \quad \left|\alpha_j^{-1} - \frac{y}{x}\right| < \frac{C_0}{H(x, y)^\mu}.
$$
It now follows from Proposition \ref{prop:automorphisms} that $\gamma \leq \frac{\Aut' |F|}{2}$. The division by $2$ appears due to the presence of the matrix $\left(\begin{smallmatrix}-1 & 0\\0 & -1\end{smallmatrix}\right)$ in $\Aut' |F|$, which maps $(x, y)$ to $(-x, -y)$.
\end{proof}

We conclude this section with the proof of Theorem \ref{thm:count_large_approximations}.

\begin{proof}[Proof of Theorem \ref{thm:count_large_approximations}]
Throughout the proof we will be adjusting our choice of $C_{16}$ four times. First, let $C_{16} \geq C_{11}$, where the positive real number $C_{11}$ is defined in Corollary \ref{cor:unique}. Then it follows from Lemma \ref{cor:unique} that for each $x/y$ satisfying (\ref{eq:roths_inequality}) the index $j \in \{1, 2, \ldots, n\}$ is unique.

Let $x_1/y_1, x_2/y_2, \ldots, x_\ell/y_\ell$ be the list of rational numbers in lowest terms that satisfy the following conditions.
\begin{enumerate}
\item $C_{16} \leq H(x_1, y_1) \leq H(x_2, y_2) \leq \ldots \leq H(x_\ell, y_\ell)$.

\item $\gcd(x_j, y_j) = 1$ for all $j = 1, 2, \ldots, \ell$.

\item For each $j \in \{1, 2, \ldots, \ell\}$, there exists the index $i_j \in \{1, 2, \ldots, n\}$ such that
$$
\left|\alpha_{i_j} - \frac{x_j}{y_j}\right| < \frac{C_0}{H(x_j, y_j)^\mu}.
$$
By the discussion above, this index is unique.

\item For every $j, k \in \{1, 2, \ldots, \ell\}$, if $\alpha_{i_k} \in \orb(\alpha_{i_j})$, i.e.,
$$
\alpha_{i_k} = \frac{s\alpha_{i_j} + t}{u\alpha_{i_j} + v}
$$
for some integers $s$, $t$, $u$ and $v$, then
$$
\frac{x_k}{y_k} \neq \frac{sx_j + ty_j}{ux_j + vy_j}.
$$
\end{enumerate}
Due to the fourth condition this list need not be uniquely defined. This fact, however, does not affect our estimates. The fourth property requires additional clarification: to each rational approximation in the list
$$
\frac{x_1}{y_1}, \quad \frac{x_2}{y_2}, \quad \ldots, \quad \frac{x_\ell}{y_\ell}
$$
correspond several rational approximations, which we call \emph{derived}. To be more precise, from $x_j/y_j$ one can naturally construct a (possibly bad) rational approximation to arbitrary $\alpha \in \orb(\alpha_{i_j})$ as follows. Let
$$
\alpha = \frac{s\alpha_{i_j} + t}{u\alpha_{i_j} + v} \quad \textrm{and} \quad \frac{x_j'}{y_j'} = \frac{sx_j + ty_j}{ux_j + vy_j}
$$
for some integers $s$, $t$, $u$ and $v$. Then
$$
\alpha - \frac{x_j'}{y_j'} = \frac{tu - sv}{(u\alpha_{i_j} + v)(u(x_j/y_j) + v)}\left(\alpha_{i_j} - \frac{x_j}{y_j}\right),
$$
so rational approximations to $\alpha$ and $\alpha_{i_j}$ are connected. Thus, by imposing condition (4), we insist that $x_j'/y_j'$ does not appear in the list $x_1/y_1, x_2/y_2, \ldots, x_\ell/y_\ell$.

In order to account for the presence of derived rational approximations, we introduce the value $\gamma_i$ defined in (\ref{eq:gamma}). Note that the value $\gamma_{i_j}$ is equal to the number of rational approximations derived from $x_j/y_j$, including $x_j/y_j$ itself. Consequently, if we let $N$ denote the total number of rationals satisfying the conditions specified in the hypothesis, then $N$ does not exceed $\sum_{j = 1}^\ell \gamma_{i_j}$. Therefore,
$$
N \leq \sum\limits_{j = 1}^\ell\gamma_{i_j} \leq \gamma\ell,
$$
where $\gamma$ is defined in (\ref{eq:gamma}). Thus, it remains to estimate $\ell$.

To derive an upper bound on $\ell$, we begin by applying a generalized gap principle to the ordered pair $(\alpha_{i_k}, \alpha_{i_{k+1}})$. Choose $C_{16}$ and define $C$ as follows:
$$
C_{16} \geq \max\limits_{j, k}\{C_1(\alpha_j, \alpha_k, \mu, C_0),\ C_3(\alpha_j, \alpha_k, \mu, C_0)\},
$$
$$
C = \max\limits_{j, k}\{C_2(\alpha_j, \alpha_k, \mu, C_0),\ C_4(\alpha_j, \alpha_k, \mu, C_0)\},
$$
where the positive real numbers $C_1$, $C_2$, $C_3$ and $C_4$ are taken from Theorems \ref{thm:archimedean_gap_principle} and \ref{thm:non-archimedean_gap_principle}, respectively. Note that if $K = \Q_p$, then $|y_k| \leq 1$, and so
$$
|y_k\alpha_{i_k} - x_k| = |y_k|\cdot \left|\alpha_{i_k} - \frac{x_k}{y_k}\right| < \frac{C_0}{H(x_k, y_k)^\mu}.
$$
Analogously,
$$
|y_{k + 1}\alpha_{i_{k + 1}} - x_{k + 1}| < \frac{C_0}{H(x_{k + 1}, y_{k + 1})^\mu}.
$$
It follows from Theorems \ref{thm:archimedean_gap_principle} and \ref{thm:non-archimedean_gap_principle} that, for every $k \in \{1, 2, \ldots, \ell - 1\}$,
$$
H(x_{k + 1}, y_{k + 1}) > C^{-1}H(x_k, y_k)^E,
$$
where
\begin{equation}
E = \mu - d/2.
\end{equation}
Notice that case 2 in the aforementioned theorems is impossible due to the fact that the list $x_1/y_1, \ldots, x_\ell/y_\ell$ does not contain derived rational approximations. Consequently,
\begin{align*}
\log H(x_\ell, y_\ell)
& > E\log H(x_{\ell - 1}, y_{\ell - 1}) - \log C\\
& > E^2\log H(x_{\ell - 2}, y_{\ell - 2}) - (1 + E)\log C\\
& > \cdots\\
& > E^{\ell - 1}\log H(x_1, y_1) - (1 + E + \cdots + E^{\ell - 2})\log C.
\end{align*}
Thus, we obtain the following lower bound on $\log H(x_\ell, y_\ell)$:
\begin{equation} \label{eq:qell_bound}
\log H(x_\ell, y_\ell) > E^{\ell - 1}\log H(x_1, y_1) - \frac{E^{\ell - 1} - 1}{E - 1}\log C.
\end{equation}

Next, we apply the Thue-Siegel principle from Lemma \ref{lem:thue-siegel} to the pair $(\alpha, \beta) = (\alpha_{i_1}, \alpha_{i_\ell})$. Observe that, since all $\alpha_i$'s have degree $d$, we have $\Q(\alpha_{i_1}) = \Q(\alpha_{i_\ell})$, so $\alpha_{i_{\ell}} \in \Q(\alpha_{i_1})$. For $a = 1/500$, set
$$
t = \sqrt{\frac{2}{d + a^2}}, \quad \tau = 2at.
$$
Then
$$
\lambda = \frac{2}{t - \tau} = \frac{2}{(1 - 2a)t} < 1.42 \sqrt d.
$$
Further,
$$
\frac{t^2}{2 - dt^2} = \frac{1}{a^2} = 500^2,
$$
$$
A_1 = 500^2\left(\log M(\alpha_{i_1}) + \frac{d}{2}\right), \quad A_\ell = 500^2\left(\log M(\alpha_{i_\ell}) + \frac{d}{2}\right),
$$
$$
\delta = \frac{dt^2 + \tau^2 - 2}{d - 1} = \frac{6a^2}{(d + a^2)(d - 1)}.
$$
Note that
\begin{equation} \label{eq:delta_inverse_inequality}
\delta^{-1} < 41667d^2.
\end{equation}
We further adjust our definition of $C_{16}$ by choosing it so that
\begin{equation} \label{eq:large_C1_third_adjustment}
C_{16} \geq C_0^{\frac{1}{\mu - 1.42\sqrt d}}\left(4e^A\right)^{\frac{1.42\sqrt d}{\mu - 1.42\sqrt d}},
\end{equation}
where $A$ is defined in (\ref{eq:A}). Now with the help of inequalities $\lambda < 1.42\sqrt d$ and  $H(x_j, y_j) \geq C_{16}$ we obtain
$$
\left|\alpha_{i_j} - \frac{x_j}{y_j}\right| < \frac{C_0}{H(x_j, y_j)^\mu} \leq \frac{1}{\left(4e^AH(x_j, y_j)\right)^{1.42 \sqrt d}}  < \frac{1}{(4e^A H(x_j, y_j))^\lambda},
$$
so that the hypothesis of Lemma \ref{lem:thue-siegel} is satisfied. Thus, we arrive at the conclusion that
\begin{align*}
\log H(x_\ell, y_\ell)
& \leq \delta^{-1}\left(\log(4e^{A_1}) + \log H(x_1, y_1)\right) - \log(4e^{A_\ell})\\
& < 41667d^2\left(\log(4e^{A_1}) + \log H(x_1, y_1)\right),
\end{align*}
where the last inequality follows from (\ref{eq:delta_inverse_inequality}). Thus
$$
\log H(x_\ell, y_\ell) < 41667d^2\left(\log\left(4e^{A_1}\right) + \log H(x_1, y_1)\right).
$$
We combine the above upper bound on $\log H(x_\ell, y_\ell)$ with the lower bound given in (\ref{eq:qell_bound}):
$$
E^{\ell - 1}\log H(x_1, y_1) - \frac{E^{\ell - 1} - 1}{E - 1}\log C < 41667d^2\left(\log\left(4e^{A_1}\right) + \log H(x_1, y_1)\right).
$$
Reordering the terms yields
\begin{equation} \label{eq:simplifying}
\left(E^{\ell - 1} - 41667d^2\right)\log H(x_1, y_1) - \frac{E^{\ell - 1} - 1}{E - 1}\log C < 41667d^2\log\left(4e^{A_1}\right).
\end{equation}
Let us assume that
$$
\ell \geq 1 + \frac{\log(41667d^2)}{\log(\mu - d/2)},
$$
for otherwise the statement of our theorem holds. Then $E^{\ell - 1} \geq 41667d^2$, so we may use the inequality $H(x_1, y_1) \geq C_{16}$ to replace $H(x_1, y_1)$ with $C_{16}$ in (\ref{eq:simplifying}):
$$
\left(E^{\ell - 1} - 41667d^2\right)\log C_{16} - \frac{E^{\ell - 1} - 1}{E - 1}\log C < 41667d^2\log\left(4e^{A_1}\right).
$$
Since $E = \mu - d/2$,
$$
(\mu - d/2)^{\ell - 1}\left(\log C_{16} - \frac{\log C}{E - 1}\right) < 41667d^2\log C_{16} + 41667d^2\log\left(4e^{A_1}\right) + \frac{\log C}{E - 1}.
$$
We make a final adjustment to $C_{16}$ by choosing it so that
\begin{equation} \label{eq:large_C1_fourth_adjustment}
C_{16} \geq C^{2/(E - 1)}.
\end{equation}
Then
$$
(\mu - 0.5d)^{\ell - 1}\frac{\log C_{16}}{2} < \left(41667d^2 + \frac{1}{2}\right)\log C_{16} + 41667d^2\log\left(4e^{A_1}\right),
$$
leading us to a conclusion
\begin{equation} \label{eq:will-continue}
(\mu - 0.5d)^{\ell - 1} < 1 + 83334d^2\left(1 + \frac{\log\left(4e^{A_1}\right)}{\log C_{16}}\right).
\end{equation}
By our choice of $C_{16}$,
$$
\log C_{16} \geq \frac{1}{\mu - 1.42\sqrt d}\log C_0 + \frac{1.42\sqrt d}{\mu - 1.42\sqrt d}\log(4e^A),
$$
which means that
$$
\frac{\log(4e^A)}{\log C_{16}} \leq \frac{\mu - 1.42\sqrt d}{1.42\sqrt d + \log C_0/\log(4e^A)} < \frac{\mu - 1.42\sqrt d}{1.42\sqrt d - 1},
$$
where the last inequality follows from the fact that $C_0 > (4e^A)^{-1}$. Plugging the above inequality into (\ref{eq:will-continue}), we obtain
$$
(\mu - 0.5d)^{\ell - 1} < 1 + 83334d^2\left(1 + \frac{\mu - 1.42\sqrt d}{1.42\sqrt d - 1}\right) = 1 + 83334d^2\frac{\mu - 1}{1.42\sqrt d - 1} \leq 1 + 98896d^{3/2}\mu,
$$
where the last inequality follows from $d \geq 3$. We conclude that
$$
\ell < 1 + \frac{\log(98897d^{3/2}\mu)}{\log(\mu - d/2)} < 1 + \frac{11.51 + 1.5\log d + \log \mu}{\log(\mu - d/2)}.
$$
The result follows once we multiply the right-hand side by the constant $\gamma$ defined in (\ref{eq:gamma}).
\end{proof}

\section*{Acknowledgements}

The author is grateful to Prof.\ Cameron L.\ Stewart for his wise supervision as well as to the anonymous referee for their excellent suggestions on how to improve the article.

\bibliography{mosunov-gap-principle}

\begin{thebibliography}{10}

\bibitem{bombieri-gubler}
E.~Bombieri and W.~Gubler.
\newblock {\em Heights in Diophantine Geometry}.
\newblock Cambridge University Press, 2006.

\bibitem{bombieri-mueller}
E.~Bombieri and J.~Mueller.
\newblock On effective measures of irrationality for $\sqrt[r]{\frac{a}{b}}$
  and related numbers.
\newblock {\em J. Reine Angew. Math.}, 342:173--196, 1983.

\bibitem{bombieri-schmidt}
E.~Bombieri and W.~M. Schmidt.
\newblock On {T}hue's equation.
\newblock {\em Invent. Math.}, 88:69--81, 1987.

\bibitem{bombieri-vaaler}
E.~Bombieri and J.~Vaaler.
\newblock On {S}iegel's lemma.
\newblock {\em Invent. Math.}, 73:11--32, 1983.

\bibitem{bugeaud-mignotte}
Y.~Bugeaud and M.~Mignotte.
\newblock On the distance between roots of integer polynomials.
\newblock {\em Proc. Edinburgh Math. Soc.}, 47:553--556, 2004.

\bibitem{curtis-reiner}
C.~Curtis and I.~Reiner.
\newblock {\em {R}epresentation {T}heory of {F}inite {G}roups and {A}ssociative
  {A}lgebras}, pages 258--262.
\newblock John Wiley \& Sons, 1962.

\bibitem{gautschi}
W.~Gautschi.
\newblock Norm estimates for inverses of {V}andermonde matrices.
\newblock {\em Numer. Math.}, 23:337--347, 1975.

\bibitem{gyory}
K.~Gy\H{o}ry.
\newblock Thue inequalities with a small number of primitive solutions.
\newblock {\em Period. Math. Hungar.}, 42(1-2):239--246, 2001.

\bibitem{lewis-mahler}
D.~Lewis and K.~Mahler.
\newblock Representation of integers by binary forms.
\newblock {\em Acta Arith.}, 6:333--363, 1961.

\bibitem{mueller-schmidt}
J.~Mueller and W.~M. Schmidt.
\newblock Thue's equation and a conjecture of siegel.
\newblock {\em Acta Math.}, 160:207--247, 1988.

\bibitem{newman}
M.~Newman.
\newblock {\em Integral Matrices}, volume~45 of {\em Pure and Appl. Math.}
\newblock Academic Press, New York, 1972.

\bibitem{prasolov}
V.~P. Prasolov.
\newblock {\em Polynomials}, volume~11 of {\em Algorithms and Computation in
  Mathematics}.
\newblock Springer-Verlag Berlin Heidelberg, 2004.

\bibitem{schmidt3}
W.~M. Schmidt.
\newblock {\em Diophantine Approximations and Diophantine Equations}.
\newblock Springer-Verlag, 1991.

\bibitem{stewart2}
C.~L. Stewart.
\newblock On the number of solutions of polynomial congruences and {T}hue
  equations.
\newblock {\em J. Amer. Math. Soc.}, 4:793--835, 1991.

\bibitem{thue}
A.~Thue.
\newblock \"{U}ber {A}nn\"aherungswerte algebraischer {Z}ahlen.
\newblock {\em J. Reine Angew. Math.}, 135:284--305, 1909.

\bibitem{zannier}
U.~Zannier.
\newblock {\em Lecture Notes on Diophantine Analysis}.
\newblock Scuola Normale Superiore Pisa, 2014.

\end{thebibliography}
\bibliographystyle{plain}
\end{document}